\documentclass[a4paper,
11pt,
pdftex,
headings=normal,
headsepline,
footsepline,
onecolumn,
headinclude,
footinclude,
DIV14,
abstracton]
{scrartcl}

\usepackage
{
	graphicx,
	amssymb,
	amsmath,
	amsthm,
	xcolor,
	dsfont,
	algpseudocode,
	authblk,
}

\usepackage[ngerman, english]{babel}

\usepackage[bf]{caption}
\usepackage[colorlinks,
pdffitwindow=false,
plainpages=false,
pdfpagelabels=true,
pdfpagemode=UseOutlines,
pdfpagelayout=SinglePage,
bookmarks=false,
colorlinks=true,
hyperfootnotes=false,
linkcolor=blue,
citecolor=green!50!black]
{hyperref}

\usepackage{enumerate}
\usepackage{algorithm}
\usepackage{subcaption}    
\usepackage{tabularx}

\newcommand{\R}{\mathbb{R}}

\newtheorem{theorem}{Theorem}[section]
\newtheorem{corollary}[theorem]{Corollary}
\newtheorem{lemma}[theorem]{Lemma}

\newtheorem{definition}[theorem]{Definition}
\newtheorem{remark}[theorem]{Remark}
\newtheorem{example}[theorem]{Example}

\let\OLDthebibliography\thebibliography
\renewcommand\thebibliography[1]{
	\OLDthebibliography{#1}
	\setlength{\parskip}{0pt}
	\setlength{\itemsep}{0pt plus 0.3ex}
}

\begin{document}

\title{On the hierarchical structure of Pareto critical sets}
\author[1]{Bennet Gebken}
\author[1]{Sebastian Peitz}
\author[1]{Michael Dellnitz}
\affil[1]{\normalsize Department of Mathematics, Paderborn University, Germany}

\maketitle

\begin{abstract}
	In this article we show that the boundary of the Pareto critical set of an unconstrained multiobjective optimization problem (MOP) consists of Pareto critical points of subproblems considering subsets of the objective functions. If the Pareto critical set is completely described by its boundary (e.g.~if we have more objective functions than dimensions in the parameter space), this can be used to solve the MOP by solving a number of MOPs with fewer objective functions. If this is not the case, the results can still give insight into the structure of the Pareto critical set. This technique is especially useful for efficiently solving many-objective optimization problems by breaking them down into MOPs with a reduced number of objective functions.
\end{abstract}

\section{Introduction}
In many applications the problem arises to optimize several functions at once. In production for example, a typical goal is to maximize the quality of a product and to minimize the production cost at the same time. If the individual goals are conflicting, there exists no single point that optimizes all objectives at once, such that scalar optimization theory cannot be applied. Instead, the goal is to compute the set of optimal compromises, the so-called \emph{Pareto set}, consisting of all \emph{Pareto optimal} points. The task of finding the Pareto set is called \emph{multiobjective optimization} (MOO).

A popular first-order necessary condition for Pareto optimality is given by Kuhn and Tucker \cite{KT1951}. It states that in a Pareto optimal point of an unconstrained MOP, there exists a convex combination of the gradients of the objective functions which is zero. Similar to scalar-valued optimization, the coefficients of this convex combination are called \emph{Karush-Kuhn-Tucker (KKT) multipliers}. Points for which such a convex combination exists are called \emph{Pareto critical}. Roughly speaking, this condition induces a (possibly set-valued) map from the standard simplex to the Pareto critical set. So the question arises whether it is possible to derive properties of the Pareto critical set from the set of KKT multipliers. In particular, we are interested in analyzing properties and computing the boundaries of the Pareto critical set. This is especially useful for MOPs with a large number of objectives, also known as \emph{many-objective optimization problems (MaOPs)}, where the Pareto critical sets can be very complicated and expensive to compute \cite{SLC11}. It is therefore of great interest to derive efficient methods to solve MaOPs, e.g.~by exploiting this structure. Until now, MaOPs are mainly being treated by evolutionary approaches \cite{ITN08}.

There already exist some results about the structure of Pareto sets. In \cite[Chapter 4]{P2017} relations between the boundary of the Pareto critical set and subsets of objectives are investigated for a special class of well-behaved objective functions. There the focus lies on the hierarchical structure of Pareto sets, meaning that every neglected objective function results in Pareto critical points that lie on the boundary of the previous problem.
In a more theoretical approach, De Melo showed that there is an open and dense subset of the set of all possible (smooth) objective functions $\mathcal{C}^{\infty}(\mathbb{R}^n,\mathbb{R}^k)$, where the Pareto critical set is a \emph{stratification} \cite{dM1976}. This means that the Pareto critical set of a generic smooth objective function is the union of submanifolds of $\mathbb{R}^n$, that the intersections of these manifolds behave nicely and that the boundaries of these manifolds are unions of lower dimensional manifolds. In other words, in a generic case the Pareto critical set is a manifold with boundaries and corners. In the case where all objective functions are convex (and there are less objective functions than there are dimensions in the parameter space) the Pareto set is diffeomorphic to a standard simplex and its facets correspond to Pareto sets where a certain number of objectives has been neglected (see \cite{S1973,SSS2012}). Lovison and Pecci extended this result in \cite{LP2014} by showing that for a dense class of smooth (nonconvex) objective functions, the (local) Pareto set is a Whitney stratification. Many solution methods for MOPs work in the objective space instead of the parameter space. Consequently, it can be of equal interest to investigate the boundary of the \emph{Pareto front}, i.e. the image of the Pareto set under the objective function. This has been done by Mueller-Gritschneider, Graeb and Schlichtmann in \cite{MGGS2009} (see also \cite{PKS2007,SIR2011}). A related approach is \emph{objective reduction} in the context of many-objective optimization, where the goal is to eliminate objective functions that are either redundant or have only minor influence the Pareto front (see e.g.~\cite{BZ2009,JCC2008,SDT2013}).

The goal of this article is to extend the results from \cite[Chapter 4]{P2017} to a much more general setting. We investigate the hierarchical structure of Pareto critical sets and study properties of the boundary. We show that the boundary of the standard simplex is mapped to a covering of the boundary of the Pareto critical set. Since at least one multiplier is zero on the boundary of the simplex, the boundary of the Pareto critical set can be calculated by omitting the objective functions corresponding to the zero multipliers. The number of functions that can be omitted depends on the rank of the Jacobian of the full objective function. 

The structure of this article is as follows: We start by giving a short introduction to MOO and concepts of topology in Section~\ref{sec:MO}. In Section~\ref{sec:StructureParetoSet} we first classify Pareto critical points by their respective KKT multipliers by differentiating between Pareto critical points with strictly positive KKT multipliers ($P_{\mathsf{int}}$) and critical points where all KKT multipliers have a zero somewhere ($P_0$). We show some results about the structure and relationships of those sets with varying regularity assumptions on our objective function. Since we will define the boundary of the Pareto critical set by properties of tangent cones, we then show some results about tangent cones of the Pareto critical set. An important technical result will be that if our MOP is regular enough, the tangent cone of the Pareto critical set is just the projection of the tangent space of the manifold of Pareto critical points, extended by their KKT multipliers, onto the parameter space. This will be used to prove the main result of Section~\ref{sec:StructureParetoSet}, which states that on the boundary of the Pareto critical set, (at least) one KKT multiplier is zero. In Section~\ref{sec:CalculationParetoSetViaSubproblems} we study how many KKT multipliers are zero, or in other words, how many components of the objective function have to be considered to compute the boundary of the Pareto critical set. The main result will be that the number of components required is equal to the maximal rank of the Jacobian on the Pareto critical set. In Section~\ref{sec:Conclusion}, we draw a conclusion and discuss further directions.

We conclude this introduction with a simple example to show the structure we want to investigate in this article. Consider the following convex MOP:
\begin{equation} \label{eq:motMOP}
\min_{x \in \mathbb{R}^2} f(x) = \min_{x \in \mathbb{R}^2} \begin{pmatrix}
f_1(x) \\
f_2(x) \\
f_3(x)
\end{pmatrix} = \min_{x \in \mathbb{R}^2} \begin{pmatrix}
(x_1 - 1)^2 + (x_2 + 1)^2 \\
x_1^2 + (x_2 - 1)^2 \\
(x_1 + 1)^2 + (x_2 + 1)^2
\end{pmatrix}.
\end{equation}
The objective functions are paraboloids, so the Pareto critical (and in this case Pareto optimal) set is given by the triangle with corners $(-1,-1)$, $(1,-1)$ and $(0,1)$. If we omit the third objective function, the Pareto optimal set of the resulting MOP is the line connecting $(-1,-1)$ and $(1,-1)$, so it is part of the boundary. In the same way we obtain the other parts of the boundary of the original Pareto critical set by omitting different objective functions. The situation is depicted in Figure \ref{fig:motivation}. The Pareto critical set of the subproblem $(f_1, f_2)$ is shown in red, $(f_1, f_3)$ in blue and $(f_2, f_3)$ in green.

\begin{figure}[ht] 
	\begin{subfigure}[t]{.5\textwidth}
		\centering
		\includegraphics[scale=0.5]{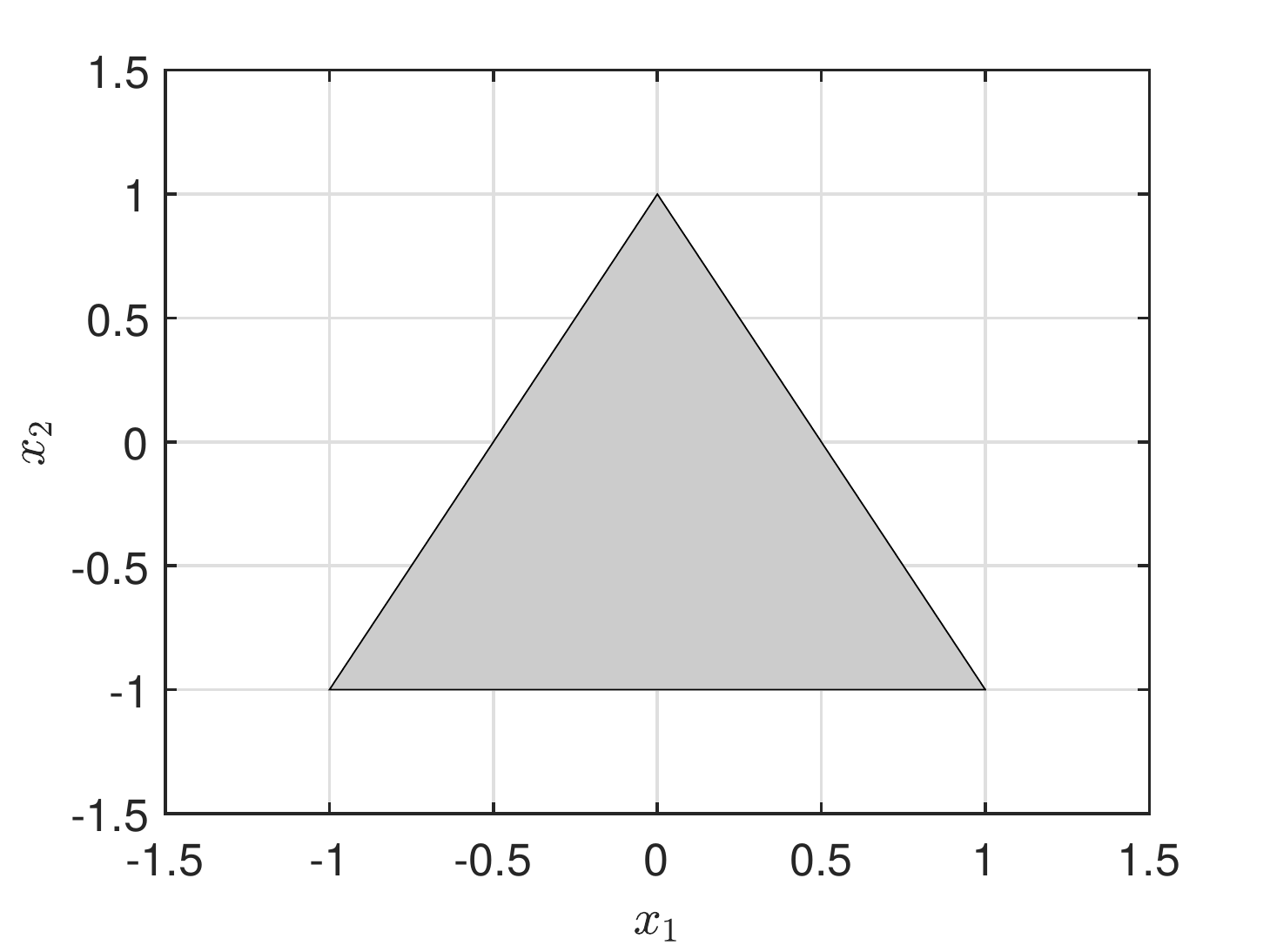}
	\end{subfigure}
	\begin{subfigure}[t]{.49\textwidth}
		\centering
		\includegraphics[scale=0.5]{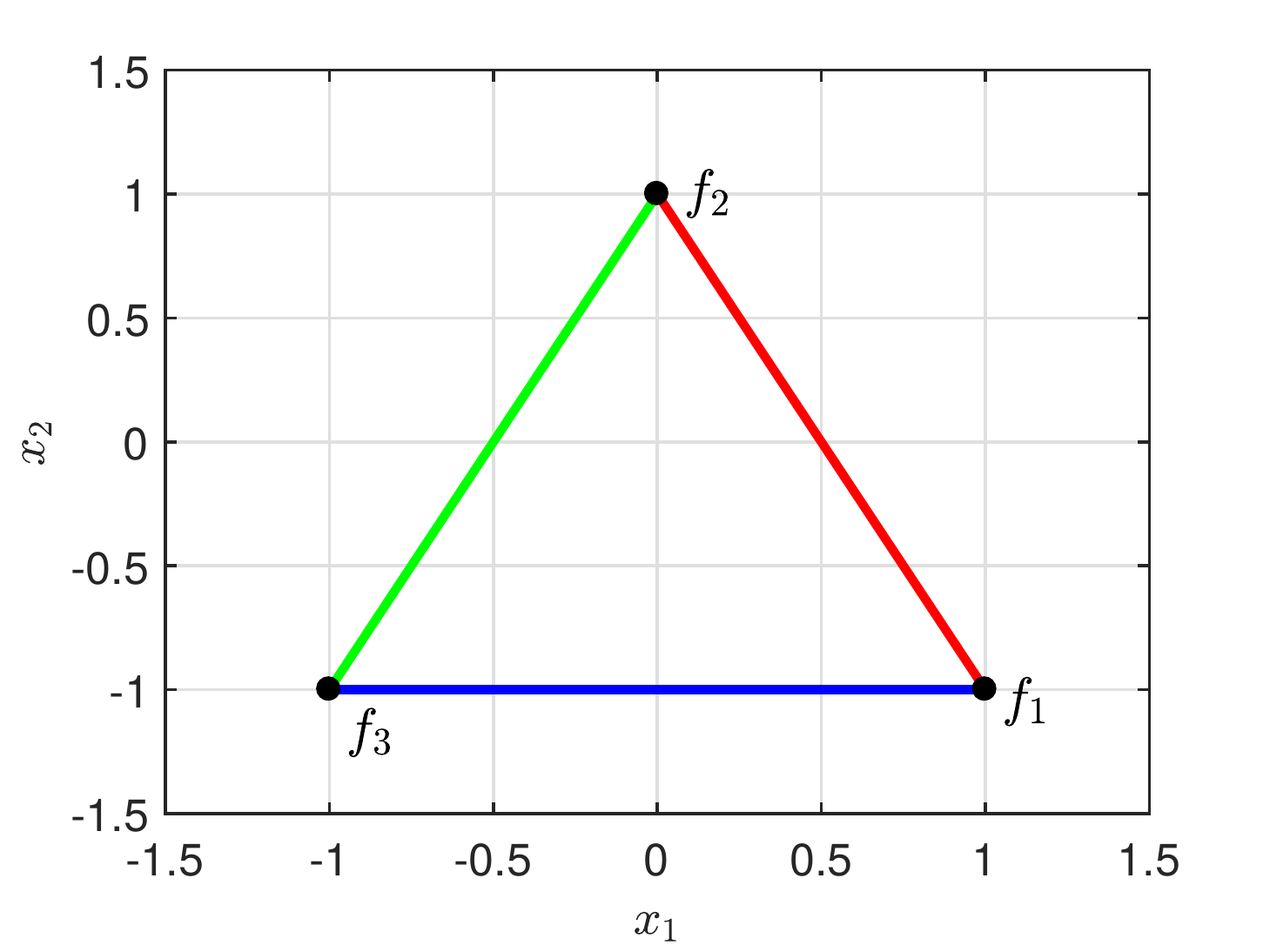}
	\end{subfigure}
	\caption{Pareto critical set of the MOP \eqref{eq:motMOP} (left) and Pareto critical sets of all 2-objective subproblems (right)}
	\label{fig:motivation}
\end{figure}

\section{Multiobjective optimization and topology} \label{sec:MO}

In this section we give a short introduction to MOO and some basic concepts of topology.

\subsection{Multiobjective optimization}

Consider the unconstrained multiobjective optimization problem
\begin{align}
	\min_{x\in\R^n} f(x) = \min_{x\in\R^n} \left( \begin{array}{c} f_1(x) \\ \vdots \\ f_k(x) \end{array} \right), \tag{MOP} \label{eq:MOP}
\end{align}
where $f: \R^n \rightarrow \R^k$ is a vector valued objective function with continuously differentiable components $f_i: \R^n \rightarrow \R$, $i = 1,...,k$.
The space of the parameters $x \in \mathbb{R}^n$ is called the \emph{parameter space} and the function $f$ is a mapping to the $k$-dimensional \emph{objective space}. In contrast to single objective optimization problems, there exists no obvious total order of the objective function values in $\mathbb{R}^k$ for $k \geq 2$ (unless the objectives are not conflicting). Therefore, the comparison of values is defined in the following way \cite{M1998}:
\begin{definition}
	Let $v, w \in \mathbb{R}^k$. The vector $v$ is \emph{less than} $w$ $(v <_p w)$, if $v_i < w_i$ for all $i \in \left\lbrace 1, \ldots, k \right\rbrace$. The relation $\leq_p$ is defined in an analogous way.
\end{definition}
A consequence of the lack of a total order is that we cannot expect to find isolated optimal points. Instead, the solution of \eqref{eq:MOP} is the set of optimal compromises, the so-called \emph{Pareto set} named after Vilfredo Pareto:
\begin{definition}
	\label{def:Pareto_optimality}
	\begin{itemize}
		\item[(a)] A point $x^* \in \R^n$ \emph{dominates} a point $x \in \R^n$, if $f(x^*) \leq_p f(x)$ and $f(x^*) \neq f(x)$.
		\item[(b)] A point $x^* \in \R^n$ is called \emph{(globally) Pareto optimal} if there exists no point $x \in \R^n$ dominating $x^*$.
		\item[(c)] The set of non-dominated points is called the \emph{Pareto set}, its image the \emph{Pareto front}.
	\end{itemize}
\end{definition}
Consequently, for each solution that is contained in the Pareto set, one can only improve one objective by accepting a trade-off in at least one other objective. A more detailed introduction to multiobjective optimization can be found in \cite{M1998,E2005}.

Similar to single objective optimization, a necessary condition for optimality is based on the gradients of the objective functions. In the multiobjective situation, the corresponding Karush-Kuhn-Tucker (KKT) condition is as follows \cite{KT1951}:
\begin{theorem} 
	Let $x^*$ be a Pareto optimal point of \eqref{eq:MOP}. Then there is some $\alpha \in (\mathbb{R}^{\geq 0})^k$ so that
	\begin{align}
		\sum_{i=1}^k \alpha_i \nabla f_i(x^*) = 0 \text{ and } \sum_{i=1}^k \alpha_i = 1 . \label{eq:MOP_optimality}
	\end{align}
\end{theorem}

Since this is only a necessary optimality condition, we introduce the following definition as a generalization of critical points in single objective optimization.

\begin{definition}
	Let $x \in \mathbb{R}^n$. If there is some $\alpha \in (\mathbb{R}^{\geq 0})^k$ such that Equation~\eqref{eq:MOP_optimality} holds, then $x$ is called \emph{Pareto critical} and $\alpha$ a \emph{KKT multiplier of} $x$. The set of Pareto critical points $P$ of \eqref{eq:MOP} is called the \emph{Pareto critical set}.
\end{definition}

The Pareto critical set contains first-order candidates for Pareto optimal points. It is the main object of interest in this article and some of its properties will be investigated in the following sections.

\subsection{Some concepts of topology}

Since we will use some topological terms in the following sections, we will introduce them for sake of completeness. 

\begin{definition} \label{def:topology}
Let $X \subseteq \mathbb{R}^n$.
\begin{enumerate}
\item[a)] $X^\circ := \{ x \in X : \exists \text{ neighborhood } U \text{ of } x \text{ with } U \subseteq X \}$ is the \emph{interior of} $X$.
\item[b)] $\overline{X} := \{ x \in \mathbb{R}^n : \exists (x_i)_i \in X \text{ with } \lim_{i \rightarrow \infty} x_i = x \}$ is the \emph{closure of} $X$.
\item[c)] $\partial X := \overline{X} \setminus X^\circ$ is the \emph{(topological) boundary of} X.
\end{enumerate}
\end{definition}

It is important to note that in general we are not looking at the topological boundary of Pareto critical sets. Instead we will look at its ``edge'' that will be defined in Section \ref{sec:StructureParetoSet}.

\section{The structure of the Pareto critical set} \label{sec:StructureParetoSet}
In this section we will investigate the structure of the Pareto critical set $P$ and extend the results from \cite{P2017}. There it was assumed that the rank of the Jacobian of the objective function $f$ has rank $k-1$ everywhere. (In particular, the assumption restricts the results to problems with $k \leq n + 1$.) Due to this, \eqref{eq:MOP} has a unique KKT multiplier $\alpha$ for every Pareto critical point $x^*$ (cf. Lemma \ref{lem:alphaUnique}) and we obtain a very nice hierarchical structure of the Pareto critical set (see Figure~\ref{fig:ExampleVierparabeln}).
\begin{figure}[ht] 
	\begin{subfigure}[t]{.33\textwidth}
		\centering
		\includegraphics[scale = 0.1]{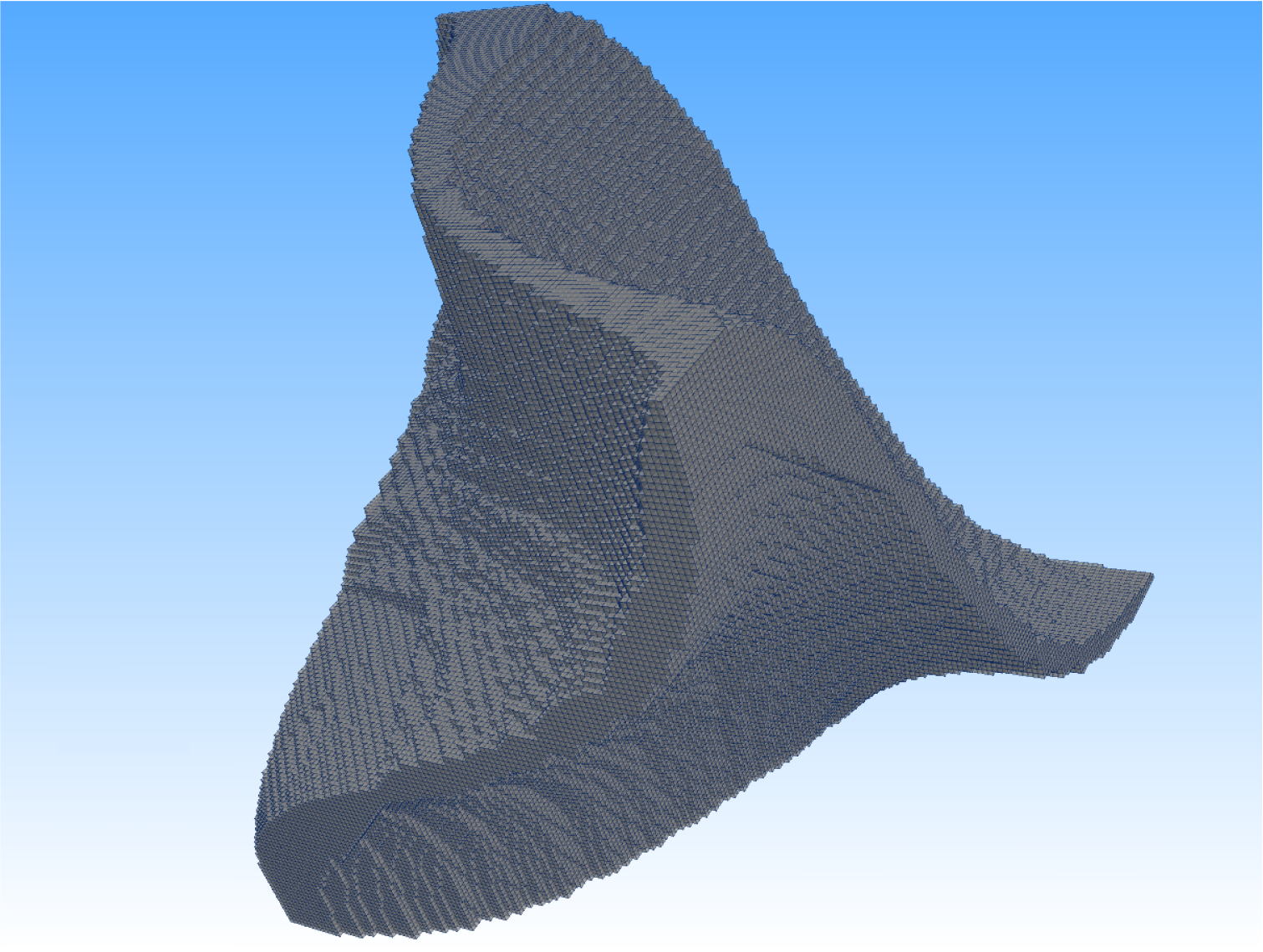}
		\caption{}
	\end{subfigure}
	\begin{subfigure}[t]{.33\textwidth}
		\centering
		\includegraphics[scale = 0.1]{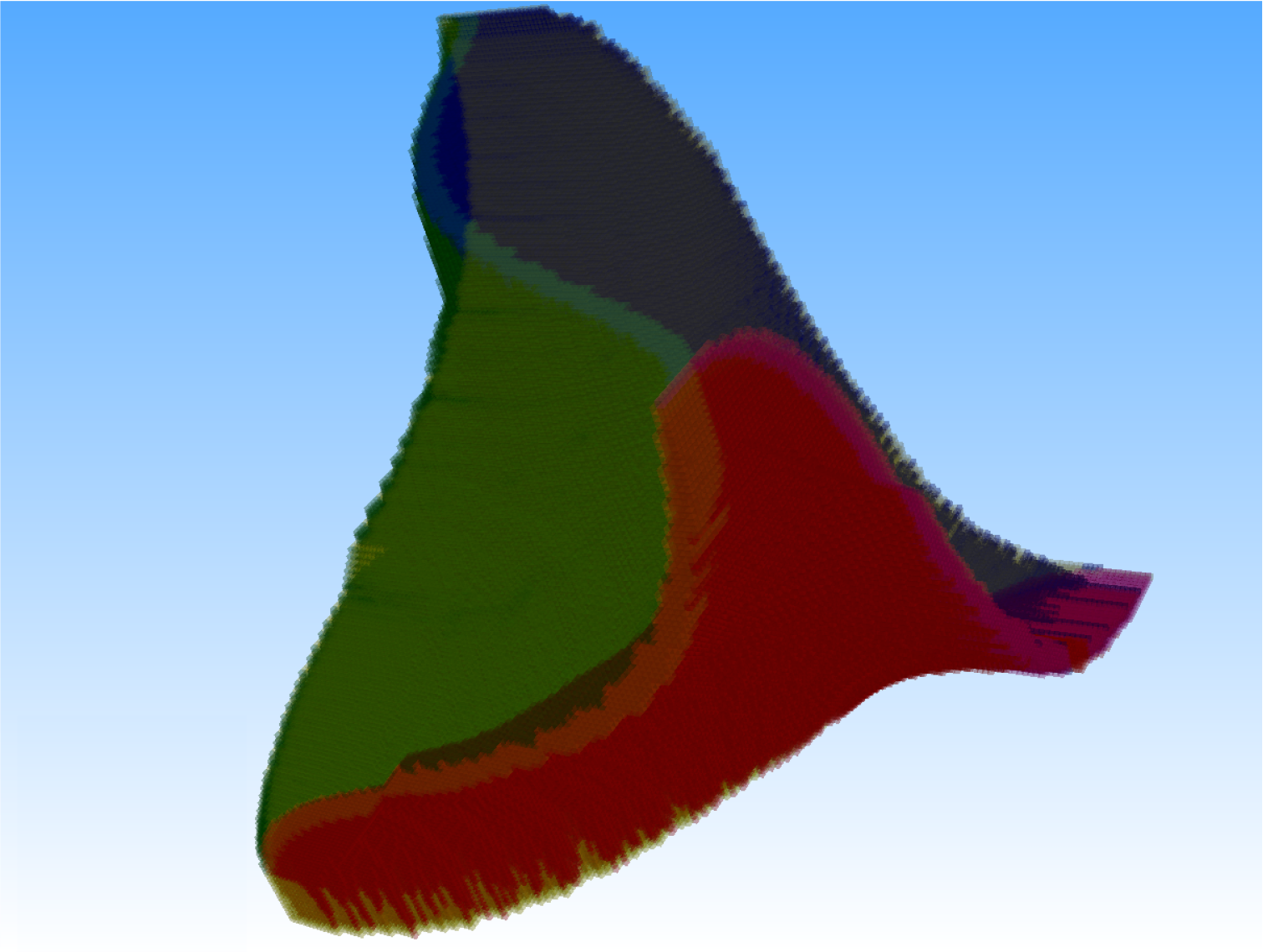}
		\caption{}
	\end{subfigure}
	\begin{subfigure}[t]{.33\textwidth}
		\centering
		\includegraphics[scale = 0.1]{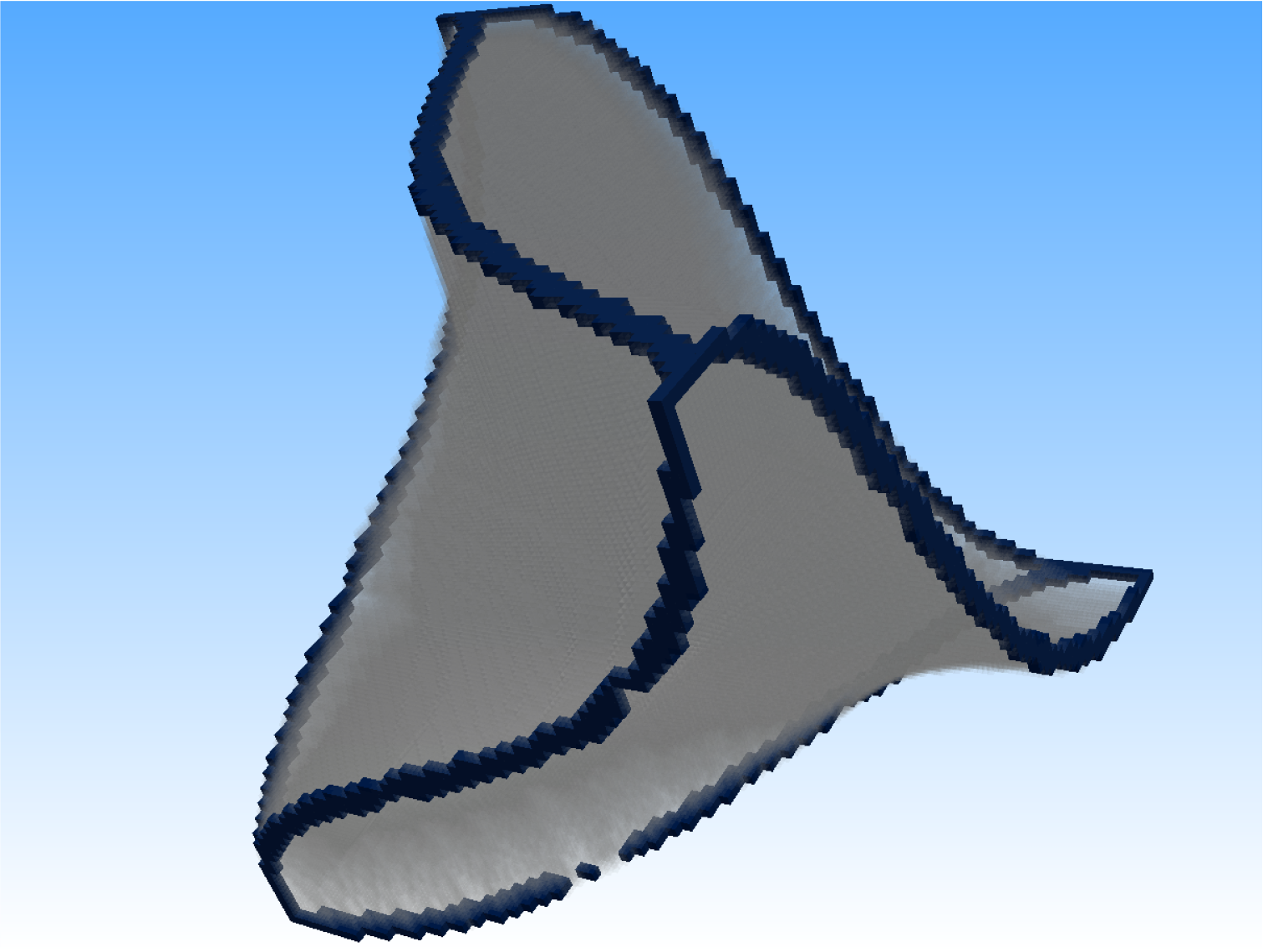}
		\caption{}
	\end{subfigure}
	\caption{(a) The Pareto set of a convex MOP with four objectives. (b) The union of the Pareto sets of the four subproblems with three objectives (shown in different colors) forms the boundary of the original Pareto set. (c) The Pareto sets of the six subproblems with two objectives are shown in blue.}
	\label{fig:ExampleVierparabeln}
\end{figure}

Here, we will generalize these results. We begin by classifying Pareto critical points by their respective KKT multipliers and then revisit the (slightly modified) result of Hillermeier \cite{H2001} about the manifold structure of the Pareto critical set extended by the set of KKT multipliers. This result will be used to show that under certain conditions, the tangent cone to the Pareto critical set is equal to the projection of the tangent space of the extended Pareto critical set onto the parameter space. This can be used to show that the edge of the Pareto set -- which is defined via the tangent cone -- is a subset of the set of Pareto critical points where (at least) one KKT multiplier is 0 (Theorem~\ref{thm:main1} on p.~\pageref{thm:main1}).

\subsection{Classifying P via KKT multipliers}

\subsubsection{$P_{\mathsf{int}}$ and $P_0$}

We begin by interpreting the KKT conditions \eqref{eq:MOP_optimality} as a nonlinear system of equations: Define $F: \mathbb{R}^n \times (\mathbb{R}^{\geq 0})^k \rightarrow \mathbb{R}^{n+1}$,
\begin{equation} \label{F}
	F(x,\alpha) := \begin{pmatrix}
		\sum_{i = 1}^k \alpha_i \nabla f_i(x) \\
		\sum_{i = 1}^k \alpha_i - 1
	\end{pmatrix}.
\end{equation}
Then $x \in \mathbb{R}^n$ is Pareto critical if there exists a KKT multiplier $\alpha \in (\mathbb{R}^{\geq 0})^k$ with $F(x,\alpha) = 0$. Let $pr_x$ be the projection onto the first $n$ components. The set of Pareto critical points $P$ is then given by
\begin{equation*}
	P = pr_x(F^{-1}( 0 )).
\end{equation*}
In order to investigate the structure of $P$, we take a closer look at the structure of the set of KKT multipliers. We distinguish between points $x \in P$ for which the corresponding KKT multipliers have at least one zero component and points that have a strictly positive KKT multiplier:
\begin{definition}\label{def:ParetoBoundaryInterior}
	For $k > 1$ let
	\begin{align*}
		P_0 := &\{ x \in P : \forall \alpha \in (\mathbb{R}^{\geq 0})^k \text{ with } F(x,\alpha) = 0 \text{ there is at least one } i \in \{1,...,k\} \\
		&\text{ with } \alpha_i = 0 \}
	\end{align*}
	and $P_{\mathsf{int}} := P \setminus P_0$.
\end{definition}
\begin{remark}
	The index $\mathsf{int}$ (for interior) has been chosen since we will show that $P_{\mathsf{int}}$ is in fact something like the ``interior'' of the Pareto critical set. Until now, $P_{\mathsf{int}}$ is only defined as the set of Pareto critical points that have a strictly positive KKT multiplier.
\end{remark}
Since for $k = 1$ we have to choose $\alpha_1 = 1$, this distinction does not make sense for single objective optimization which is why we assume $k > 1$ from now on.

\subsubsection{Geometrical properties}
The remainder of this section is dedicated to the question how $P_0$ and $P_{\mathsf{int}}$ are related to differential geometrical and topological properties of $P$.
We start by simplifying $F$. Obviously, if $F(x,\alpha) = 0$ we must have $\alpha_i = 1 - \sum_{j \neq i} \alpha_j$ for all $i \in \{1,...,k\}$, i.e. it suffices to consider $k-1$ entries of $\alpha$. Define
\begin{equation*}
	\Delta^{k-1} := \left\{ \alpha \in (\mathbb{R}^{>0})^{k-1} : \sum_{i = 1}^{k-1} \alpha_i < 1 \right\} 
\end{equation*}
with the closure
\begin{equation*}
	\overline{\Delta^{k-1}} = \left\{ \alpha \in (\mathbb{R}^{\geq 0})^{k-1} : \sum_{i = 1}^{k-1} \alpha_i \leq 1 \right\}
\end{equation*}
and $\tilde{F}: \mathbb{R}^n \times \overline{\Delta^{k-1}} \rightarrow \mathbb{R}^n$ via
\begin{align*}
	\tilde{F}(x,\alpha) &:= \sum_{i = 1}^{k - 1} \alpha_i \nabla f_i(x) + (1 - \sum_{i = 1}^{k-1} \alpha_i) \nabla f_k(x) \\
	&= \sum_{i = 1}^{k-1} \alpha_i (\nabla f_i(x) - \nabla f_k(x)) + \nabla f_k(x).
\end{align*}
Then $P_{\mathsf{int}} = pr_x((\tilde{F}|_{\mathbb{R}^n \times \Delta^{k-1}})^{-1}(0))$ by Definition~\ref{def:ParetoBoundaryInterior}. Since we want to differentiate $\tilde{F}$ (e.g.~to be able to apply the Implicit Function Theorem), we assume from now on that $f$ is twice continuously differentiable. Then the first derivatives of $\tilde{F}$ are
\begin{align*}
	&D\tilde{F}(x,\alpha) = \begin{pmatrix}
	D_x \tilde{F}(x,\alpha), & D_\alpha \tilde{F}(x,\alpha) 
	\end{pmatrix}
	\in \mathbb{R}^{n \times (n+k-1)}
\end{align*}
with
\begin{equation*}
	D_x \tilde{F}(x,\alpha) = \sum_{i = 1}^{k-1} \alpha_i \nabla^2 f_i(x) + (1 - \sum_{i = 1}^{k-1} \alpha_i) \nabla^2 f_k(x) \in \mathbb{R}^{n \times n}
\end{equation*}
and
\begin{equation*}
	D_\alpha \tilde{F}(x,\alpha) =  \begin{pmatrix} \nabla f_1(x) - \nabla f_k(x), \cdots, \nabla f_{k-1}(x) - \nabla f_k(x) \end{pmatrix}  \in \mathbb{R}^{n \times (k-1)}.
\end{equation*}
By construction $x$ is Pareto critical iff there exists some $\alpha \in \overline{\Delta^{k-1}}$ with $\tilde{F}(x,\alpha) = 0$. For easier notation we define by
\begin{equation} \label{eq:defA}
	A(x) := \{ \alpha \in \overline{\Delta^{k-1}} : \tilde{F}(x,\alpha) = 0 \}
\end{equation}
the set of valid (reduced) KKT multipliers for a fixed point $x\in P$. Consequently, $x$ is Pareto critical iff $A(x) \neq \emptyset$. 

The following lemma is a first topological result about the relation between $P_{\mathsf{int}}$ and $P$.
\begin{lemma} \label{lem:clIntP}
	If $D_x \tilde{F}(x,\alpha)$ is invertible for all $x \in P$ and $\alpha \in A(x)$, then 
	\begin{equation*}
		\overline{P_{\mathsf{int}}} = P
	\end{equation*}
	where $\overline{P_{\mathsf{int}}}$ is the closure of $P_{\mathsf{int}}$ in $\mathbb{R}^n$.
\end{lemma}
\begin{proof}
	If $P = \emptyset$ we also have $P_{\mathsf{int}} = \emptyset$, so the assertion holds. Let $P \neq \emptyset$, $\bar{x} \in P$ and $\bar{\alpha} \in A(\bar{x})$. By our assumption $D_x \tilde{F}(\bar{x},\bar{\alpha})$ is invertible, so we can apply the Implicit Function Theorem to obtain neighborhoods $U \subseteq \mathbb{R}^{k-1}$ of $\bar{\alpha}$, $V \subseteq \mathbb{R}^n$ of $\bar{x}$ and a $C^1$-function $\phi : U \rightarrow V$ with $\tilde{F}(x,\alpha) = 0 \Leftrightarrow \phi(\alpha) = x$ for all $(x,\alpha) \in V \times U$. Now choose some sequence $(\alpha_i)_i \subseteq \Delta^{k-1} \cap U$ with $\alpha_i \rightarrow \bar{\alpha}$ and define $x_i := \phi(\alpha_i)$ $\forall i \in \mathbb{N}$. Since $(\alpha_i)_i \subseteq \Delta^{k-1}$ we have $(x_i)_i \subseteq P_{\mathsf{int}}$ and by continuity of $\phi$ we have $\lim_{i \rightarrow \infty} x_i = \lim_{i \rightarrow \infty} \phi(\alpha_i) = \phi(\bar{\alpha}) = \bar{x}$, hence $\bar{x} \in \overline{P_{\mathsf{int}}}$. It follows that $P \subseteq \overline{P_{\mathsf{int}}}$ and in particular $P_{\mathsf{int}} \neq \emptyset$. \\
	Let $(x_i)_i \subseteq P_{\mathsf{int}}$ be a sequence that converges to some $\bar{x}$ in $\mathbb{R}^n$. Then $(x_i)_i$ induces a sequence $(\alpha_i)_i$ with $\alpha_i \in A(x_i)$ for all $i \in \mathbb{N}$ and since $\Delta^{k-1}$ is bounded we can assume w.l.o.g.~that $\alpha_i \rightarrow \bar{\alpha} \in \overline{\Delta^{k-1}}$. Since $\tilde{F}$ is continuous we must have $\tilde{F}(\bar{x},\bar{\alpha}) = 0$ and  $\bar{x} \in P$. Thus $\overline{P_{\mathsf{int}}} \subseteq P$.
\end{proof}
If the premise of Lemma \ref{lem:clIntP} holds, $P_0$ can be thought of as a ``null set'' in the set $P$ or a set with a ``lower dimension'' than $P$. In a more general case this does not hold as we will see later (Example~\ref{ex:isolTan}). From now on assume that $P \neq \emptyset$. In particular, Lemma \ref{lem:clIntP} shows that $P_{\mathsf{int}} \neq \emptyset$ if $P \neq \emptyset$, so this assumption makes sure that we do not have to worry about the existence of points in $P_{\mathsf{int}}$.

We will now shift our view from topological to differential geometrical properties of $P$. An important and well-known result was given by Hillermeier \cite{H2001}, where it is shown that $\mathcal{M} := (\tilde{F}|_{\mathbb{R}^n \times \Delta^{k-1}})^{-1}(0)$ (with $pr_x(\mathcal{M}) = P_{\mathsf{int}}$) is a manifold.
\begin{theorem} \label{thm:Mmanifold}
Let $\mathcal{M} := (\tilde{F}|_{\mathbb{R}^n \times \Delta^{k-1}})^{-1}(0)$. If $rk(D \tilde{F}(x,\alpha)) = n$ for all $(x,\alpha) \in \mathcal{M}$, then $\mathcal{M}$ is a $(k-1)$-dimensional $C^2$-submanifold of $\mathbb{R}^{n+k-1}$ and $T_{(x,\alpha)} \mathcal{M} = ker(D \tilde{F}(x,\alpha))$. 
\end{theorem}
\begin{proof}
By our assumption $0$ is a regular value of $\tilde{F}|_{\mathbb{R}^n \times \Delta^{k-1}}$. The assertion follows from the submersion theorem (see e.g. \cite{L2012}, Corollary 5.24 and Lemma 5.29).
\end{proof}

The case that is important for what we want to investigate in this article is the special case where $rk(D_x \tilde{F}(x,\alpha)) = n$ for all $(x,\alpha) \in \mathcal{M}$. This means that we can apply the Implicit Function Theorem to locally get a differentiable mapping between the set of KKT multipliers and the Pareto critical set. If the premise of Theorem \ref{thm:Mmanifold} holds but $rk(D_x \tilde{F}(x,\alpha)) < n$, $\mathcal{M}$ is a manifold but we are missing this local relationship. To show what can happen when the premise of Theorem \ref{thm:Mmanifold} fails to hold consider the following example.

\begin{example} \label{ex:MNoManifold}
	Consider the MOP $min_{x \in \mathbb{R}^2} f(x)$ with
	\begin{equation*}
		f(x) := \begin{pmatrix}
			-2 x_1 x_2 - 2 x_1^2 - 2 x_2 + x_2^2 \\
			x_1 x_2 + x_1^2 + x_2
		\end{pmatrix}.
	\end{equation*}
	For this problem, $\mathcal{M}$ can be calculated analytically:
	\begin{equation*}
		\mathcal{M} = (\mathbb{R} \times \{ 0 \} \times \{ 1/3 \}) \cup \left\{ \left(\frac{1 - 3 \alpha}{7 \alpha - 1}, -2 \frac{1 - 3 \alpha}{7 \alpha - 1}, \alpha \right)^{\top} : \alpha \in (0,1), \alpha \neq \frac{1}{7} \right\}.
	\end{equation*}
	We see in Figure \ref{fig:MNoManifold}~(a) that $\mathcal{M}$ is the union of three smooth curves and that there is an intersection at $(0,0,1/3)^{\top}$. Since the resulting ``cross'' at $(0,0,1/3)^{\top}$ is not diffeomorphic to an open subset of $\mathbb{R}$, $\mathcal{M}$ is not a manifold. With the same argument it follows that
	\begin{equation*}
		P_{\mathsf{int}} = pr_x(\mathcal{M}) = (\mathbb{R} \times \{ 0 \}) \cup \{ (t,-2t)^{\top} : t \in \mathbb{R} \setminus [-1,-1/3] \}
	\end{equation*} 
	is not a manifold (with a cross at $(0,0)^{\top}$), cf.~Figure~\ref{fig:MNoManifold}~(b).
\end{example}

\begin{figure}[ht] 
	\begin{subfigure}[t]{.5\textwidth}
		\centering
		\includegraphics[scale=0.55]{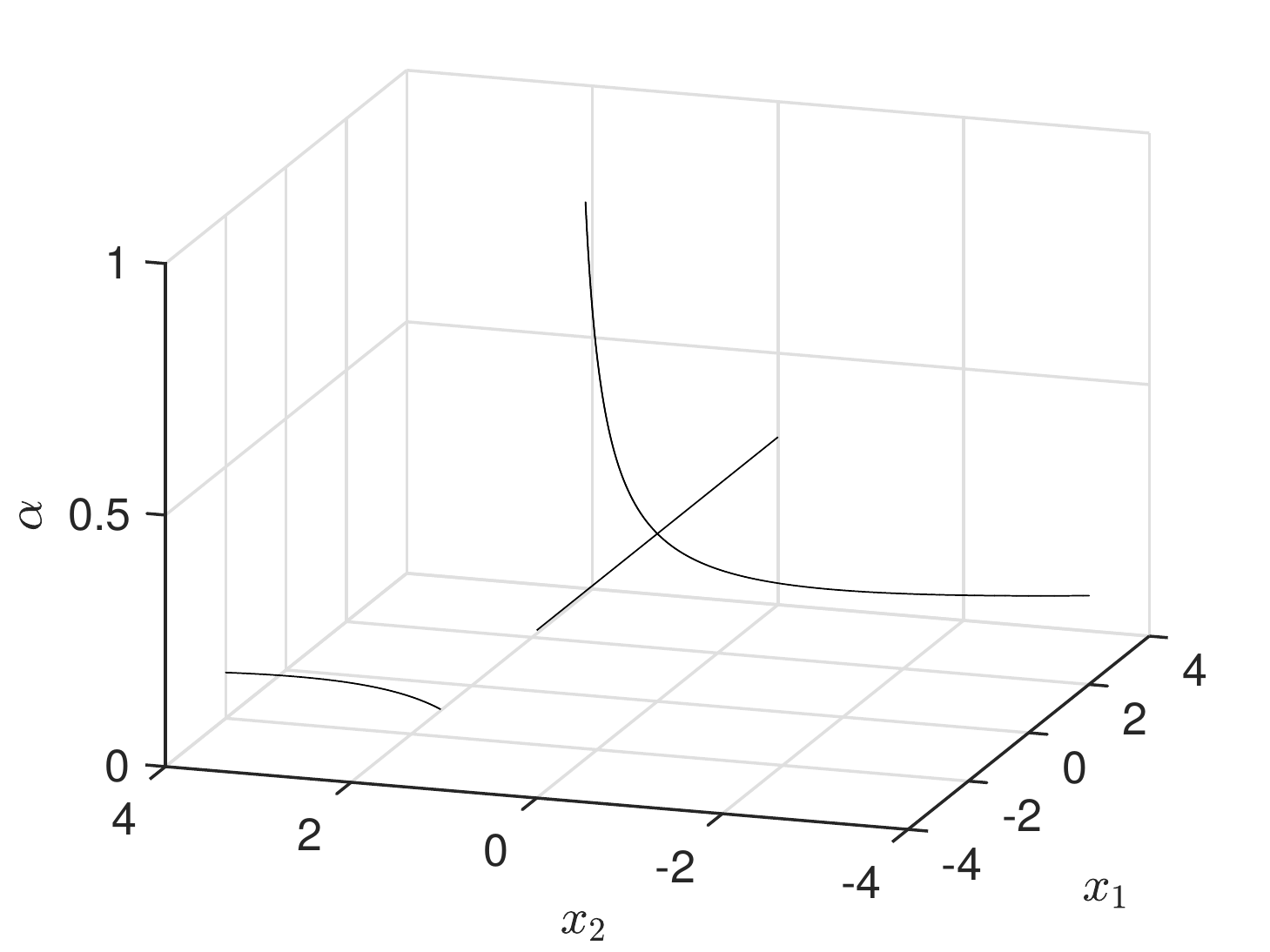}
		\caption{$\mathcal{M}$}
	\end{subfigure}
	\begin{subfigure}[t]{.49\textwidth}
		\centering
		\includegraphics[scale=0.55]{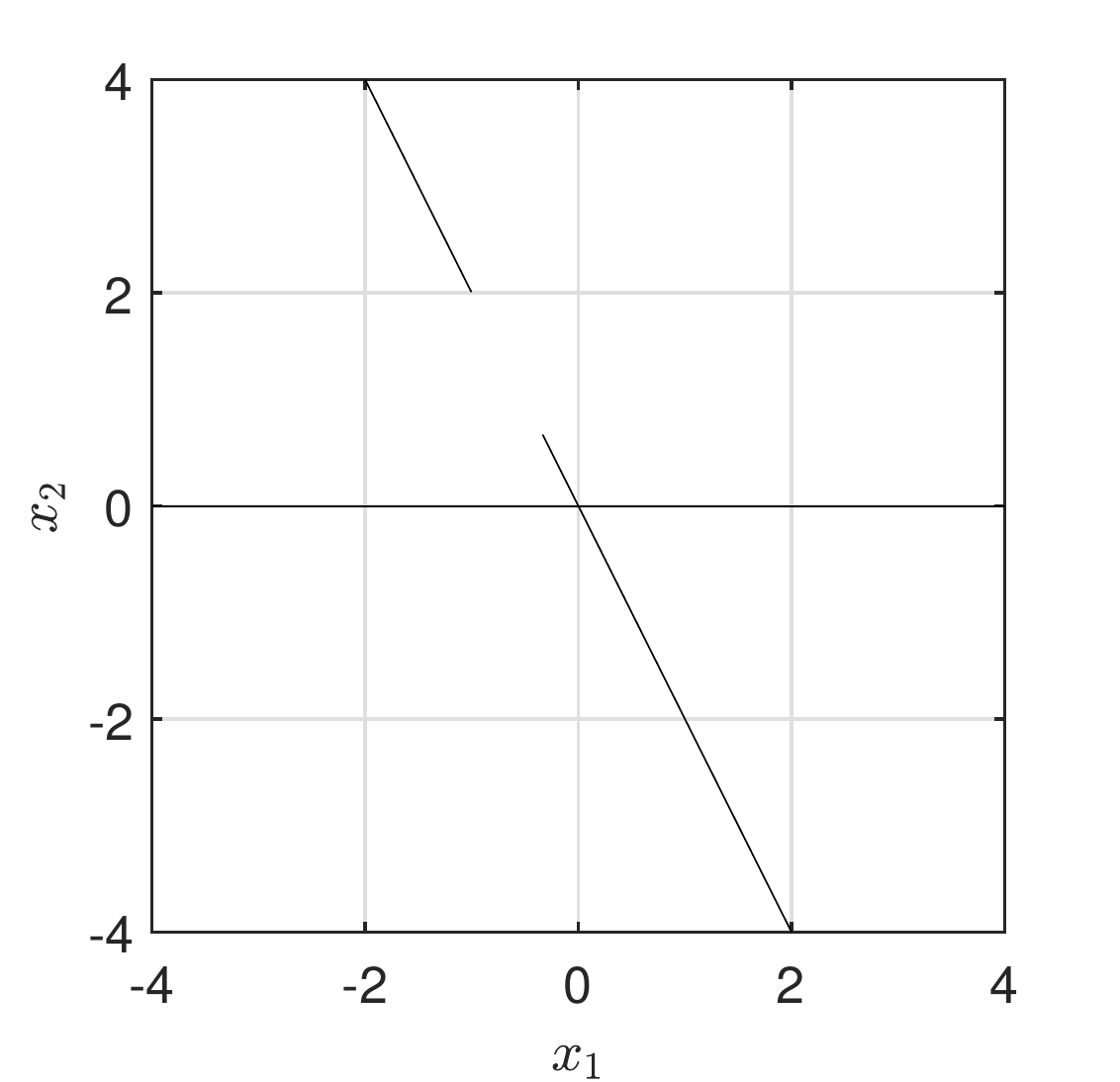}
		\caption{$P_{\mathsf{int}}$}
	\end{subfigure}
	\caption{$\mathcal{M}$ and $P_{\mathsf{int}}$ for Example \ref{ex:MNoManifold}.}
	\label{fig:MNoManifold}
\end{figure}

Observe that in Example \ref{ex:MNoManifold}, $\mathcal{M}$ could be made a manifold by removing a single point (in this case $(0,0,1/3)^{\top}$) from $\mathcal{M}$. 
This is made precise in the following lemma.
\begin{lemma} \label{lem:localManifold}
	Let $N := \{ (x,\alpha) \in \mathcal{M} : rk(D_x \tilde{F}(x,\alpha)) \leq n - 1 \}$. Then $\mathcal{M} \setminus N$ is a $(k-1)$-dimensional $C^2$-submanifold of $\mathbb{R}^{n+k-1}$ and $T_{(x,\alpha)} (\mathcal{M} \setminus N) = ker(D \tilde{F}(x,\alpha))$.
\end{lemma}
\begin{proof}
	$N$ is closed since $rk \circ D_x \tilde{F}$ is lower semicontinuous as the composition of the lower semicontinuous function $rk: \mathbb{R}^{n \times n} \rightarrow \mathbb{N}$ (see e.g. \cite{L2013}) and the continuous function $D_x \tilde{F}$. This means $U := (\R^n \times \Delta^{k-1}) \setminus N$ is open and we can apply the Submersion Theorem to $\tilde{F}|_{U}$ as in the proof of Theorem \ref{thm:Mmanifold}.
\end{proof}


\begin{remark}
	It would be sufficient to remove the set $N' := \{ (x,\alpha) \in \mathcal{M} : rk(D \tilde{F}(x,\alpha)) \leq n - 1 \} \subseteq N$ from $\mathcal{M}$ if one is only interested in having a manifold structure on $\mathcal{M}$. This follows with the proof of Lemma \ref{lem:localManifold} by replacing $N$ with $N'$. But since we want to be able to apply the Implicit Function Theorem as described above, we remove $N$ instead of $N'$.
\end{remark}

In other words, if the set of points violating the rank condition $rk(D_x \tilde{F}(x,\alpha)) = n$ is removed, the remaining set is a manifold. This means that $\mathcal{M}$ is \emph{locally} a manifold in all points that satisfy the rank condition. This has also been observed by Hillermeier \cite{H2001} in a similar way.
 
Note that the previous results on manifolds only hold in the augmented $(x,\alpha)$ space. The following example shows that even if $rk(D_x \tilde{F}(x,\alpha)) = n$ everywhere, $P_{\mathsf{int}}$ does not have to be a manifold:
\begin{example} \label{ex:PnotManifold}
	Consider the MOP $\min_{x \in \mathbb{R}^2} f(x)$ with
	\begin{equation*}
		f(x) := \begin{pmatrix}
			x_1^2 + x_2^2 \\
			(x_1 - 1)^2 + (x_2 - 1)^4 \\
			(x_1 - 2)^2 + (x_2 - 2)^2
		\end{pmatrix}.
	\end{equation*}
	The Pareto critical set can be calculated analytically and is shown in Figure \ref{fig:PnotManifold}. $P_{\mathsf{int}}$ is given by the gray area united with the points $(1,1)^{\top}$, $(1 - \frac{1}{\sqrt{2}},1 - \frac{1}{\sqrt{2}})^{\top} \approx (1.7071, 1.7071)^{\top}$ and $(1 + \frac{1}{\sqrt{2}},1 + \frac{1}{\sqrt{2}})^{\top} \approx (0.2929, 0.2929)^{\top}$. Due to these additional points, $P_{\mathsf{int}}$ is not a manifold even though $\mathcal{M}$ is.
	\begin{figure}
		\centering
		\includegraphics[scale=0.6]{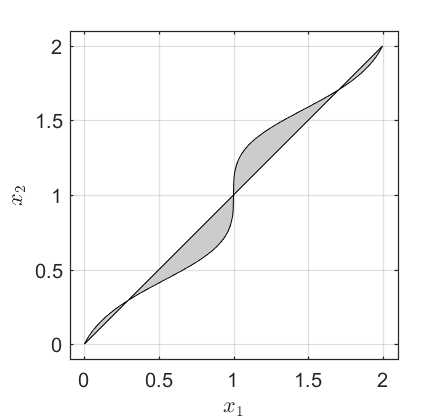}
		\caption{Pareto critical set for Example \ref{ex:PnotManifold}.}
		\label{fig:PnotManifold}
	\end{figure}
\end{example}

\subsection{Tangent cones and the uniqueness of KKT multipliers}
The goal of this section is to show that $P_0$ contains the ``boundary'' or ``edge'' of $P$. To this end, we will from now on assume that 
\begin{equation} \label{eq:assum_DxF_reg}
	rk(D_x \tilde{F}(x,\alpha)) = n \text{ for all } (x,\alpha) \in \mathcal{M},
\end{equation}
such that that $\mathcal{M}$ is a manifold with the properties stated in Theorem \ref{thm:Mmanifold}. By ``edge'' we mean the set of points in $P$ that have a tangent vector (with respect to $P$) whose negation is not a tangent vector. It will be formally defined at a later point (Definition \ref{def:P_E}). Tangent vectors are elements of the \emph{tangent cone} which is defined as follows in the general case:

\begin{definition} \label{def:tanCone}
	Let $Y \subseteq \mathbb{R}^n$ and $x \in \mathbb{R}^n$. Then
	\begin{equation*}
		Tan(Y,x) := \left\{ v \in \mathbb{R}^n : \exists (v_i)_i \subseteq \mathbb{R}^n \setminus \{ 0 \}: v_i \rightarrow 0, x + v_i \in Y, \frac{v_i}{\| v_i \|} \rightarrow \frac{v}{\| v \|}  \right\} \cup \{ 0 \}
	\end{equation*}
	is the \emph{tangent cone of} $Y$ \emph{at} $x$.
\end{definition}

The tangent cone $Tan(Y,x)$ contains the directions in $x$ that point into or alongside $Y$. Note that in contrast to the tangent space of a manifold, we do not need any structure to define the tangent cone of a set. The following lemma will be used in a later proof and shows that if $D_x \tilde{F}(x,\alpha)$ is also invertible for $x \in P$ and $\alpha \in \partial A(x)$ (the boundary of $A(x)$ in $\mathbb{R}^{k-1}$), then there are no ``isolated'' directions in $Tan(P,x)$.
\begin{lemma} \label{lem:clIntPTan}
	If $D_x \tilde{F}(x,\alpha)$ is invertible for all $x \in P$ and $\alpha \in A(x)$ then 
	\begin{equation*}
		Tan(P_{\mathsf{int}},x_0) = Tan(P,x_0) \quad \forall x_0 \in P_{\mathsf{int}}.
	\end{equation*}
\end{lemma}
\begin{proof}
	We obviously have $Tan(P_{\mathsf{int}},x_0) \subseteq Tan(P,x_0)$ since $P_{\mathsf{int}} \subseteq P$. Let $v \in Tan(P,x_0) \setminus \{ 0 \}$. Then there exists a sequence $(v_i)_i \in \mathbb{R}^n \setminus \{ 0 \}$ as in Definition \ref{def:tanCone}. Let $d_i := \| v \| \frac{v_i}{\| v_i \|}$ and $t_i = \frac{\| v_i \|}{\| v \|}$. Then $x_0 + t_i d_i \in P$ $\forall i \in \mathbb{N}$. Since $\overline{P_{\mathsf{int}}} = P$ by Lemma \ref{lem:clIntP}, we have a sequence $(z_i)_i$ with $\| z_i - d_i \| \rightarrow 0$ and $x_0 + t_i z_i \in P_{\mathsf{int}}$. Define $w_i := t_i z_i$. Then $x_0 + w_i \in P_{\mathsf{int}}$, $w_i = (z_i - d_i + d_i) t_i = (z_i - d_i) t_i + t_i d_i \rightarrow 0$ and $\frac{w_i}{\| w_i \|} = \frac{z_i}{\| z_i \|}$. Finally,
	\begin{equation*}
		\lim_{i \rightarrow \infty} \frac{d_i}{\| d_i \|} = \lim_{i \rightarrow \infty} \frac{v_i}{\| v_i \|} = \frac{v}{\| v \|},
	\end{equation*}
	and 
	\begin{align*}
		\left\| \frac{v}{\| v \|} - \frac{z_i}{\| z_i \|} \right\| 
		&\leq \left\| \frac{v}{\| v \|} - \frac{d_i}{\|d_i\|} \right\| + \left\| \frac{d_i}{\|d_i\|} -\frac{z_i}{\|z_i\|} \right\| \\
		&= \left\| \frac{v}{\| v \|} - \frac{d_i}{\|d_i\|} \right\| + \frac{1}{\| d_i \| \| z_i \|} \Big\| \|z_i\| d_i - \|d_i\| z_i \Big\| \\
		&\leq \left\| \frac{v}{\| v \|} - \frac{d_i}{\|d_i\|} \right\| + \frac{1}{\| d_i \| \| z_i \|} \Big( \|z_i\| \|d_i - z_i\| + \|z_i\| \Big( \|z_i\| - \|d_i\| \Big) \Big).
	\end{align*}
	Consequently,
	\begin{equation*}
		\lim_{i \rightarrow \infty} \frac{z_i}{\|z_i\|} = \frac{v}{\| v \|},
	\end{equation*}
	since $\lim_{i \rightarrow \infty} \| d_i \| > 0$ and $\| z_i - d_i \| \rightarrow 0$.
\end{proof}

In order to show the irregularities that may occur in MOPs when $D_x \tilde{F}(x,\alpha)$ is not invertible for $x \in P$ and $\alpha \in \partial A(x)$, we will now give an example where Lemma~\ref{lem:clIntPTan} (and Lemma~\ref{lem:clIntP}) does not hold.

\begin{example} \label{ex:isolTan}
	Consider the MOP  $\min_{x \in \mathbb{R}^2} f(x)$ with 
	\begin{equation*}
		f(x) := \begin{pmatrix}
			x_1^2 + (x_2 - 1)^2 \\
			x_1^2 + (x_2 + 1)^2 \\
			x_2^2 
		\end{pmatrix}, \quad
		Df(x) = \begin{pmatrix}
			2 x_1 & 2 (x_2 - 1) \\
			2 x_1 & 2 (x_2 + 1) \\
			0     & 2 x_2 
		\end{pmatrix}.
	\end{equation*}
	Let $F$ be as in (\ref{F}), i.e. 
	\begin{align*}
		F(x,\alpha) &= \begin{pmatrix}
			2 x_1 \alpha_1 + 2 x_1 \alpha_2 \\
			2(x_2 - 1) \alpha_1 + 2(x_2 + 1) \alpha_2 + 2 x_2 \alpha_3 \\
			\alpha_1 + \alpha_2 + \alpha_3 - 1
		\end{pmatrix} \\
		&= \begin{pmatrix}
			2 x_1 (\alpha_1 + \alpha_2) \\
			2 x_2 (\alpha_1 + \alpha_2 + \alpha_3) + 2 (\alpha_2 - \alpha_1) \\
			\alpha_1 + \alpha_2 + \alpha_3 - 1
		\end{pmatrix}.
	\end{align*}
	It is easy to see that $P = (\mathbb{R} \times \{ 0 \}) \cup (\{ 0 \} \times [-1,1])$. We will now determine $P_{\mathsf{int}}$ and $P_0$ for this MOP:
	\begin{itemize}
		\item $x_1 \neq 0$: $F(x,\alpha) = 0$ iff $x \in \mathbb{R} \times \{ 0 \}$ and $\alpha = (0,0,1)^{\top}$.
		\item $x_1 = 0$: $F(x,\alpha) = 0$ iff $x \in \{ 0 \} \times [-1,1]$ and $\alpha \in \{ (\lambda_1,\lambda_1 - x_2,1 - 2 \lambda_1 + x_2)^{\top} : \lambda_1 \in [0,1] \cap [x_2,1+x_2] \cap [\frac{x_2}{2},\frac{1+x_2}{2}] \}$. 
	\end{itemize}
	Thus, $P_{\mathsf{int}} = \{ 0 \} \times (-1,1)$ and $P_0 = (\mathbb{R} \times \{ 0 \}) \setminus \{(0,0)^{\top}\} \cup \{(0,1)^{\top}, (0,-1)^{\top} \}$, as depicted in Figure \ref{fig:isolTan}. Consequently, $\overline{P_{\mathsf{int}}} = \{ 0 \} \times [-1,1] \neq P$. In $(0,0)^{\top}$ we have $Tan(P_{\mathsf{int}},(0,0)^{\top}) = \{ 0 \} \times \mathbb{R}$ and $Tan(P,(0,0)^{\top}) = ( \{ 0 \} \times \mathbb{R} ) \cup ( \mathbb{R} \times \{ 0 \} )$, hence $Tan(P_{\mathsf{int}},(0,0)^{\top}) \neq Tan(P,(0,0)^{\top})$. The reason for this is that $\nabla^2 f_3(x) = \begin{pmatrix}
		0 & 0 \\
		0 & 2
	\end{pmatrix}$, 
	so $D_x \tilde{F}((0,0)^{\top},(0,0,1)^{\top}) = \nabla^2 f_3(x)$ is not invertible.
	\begin{figure}
		\centering
		\includegraphics[scale=0.7]{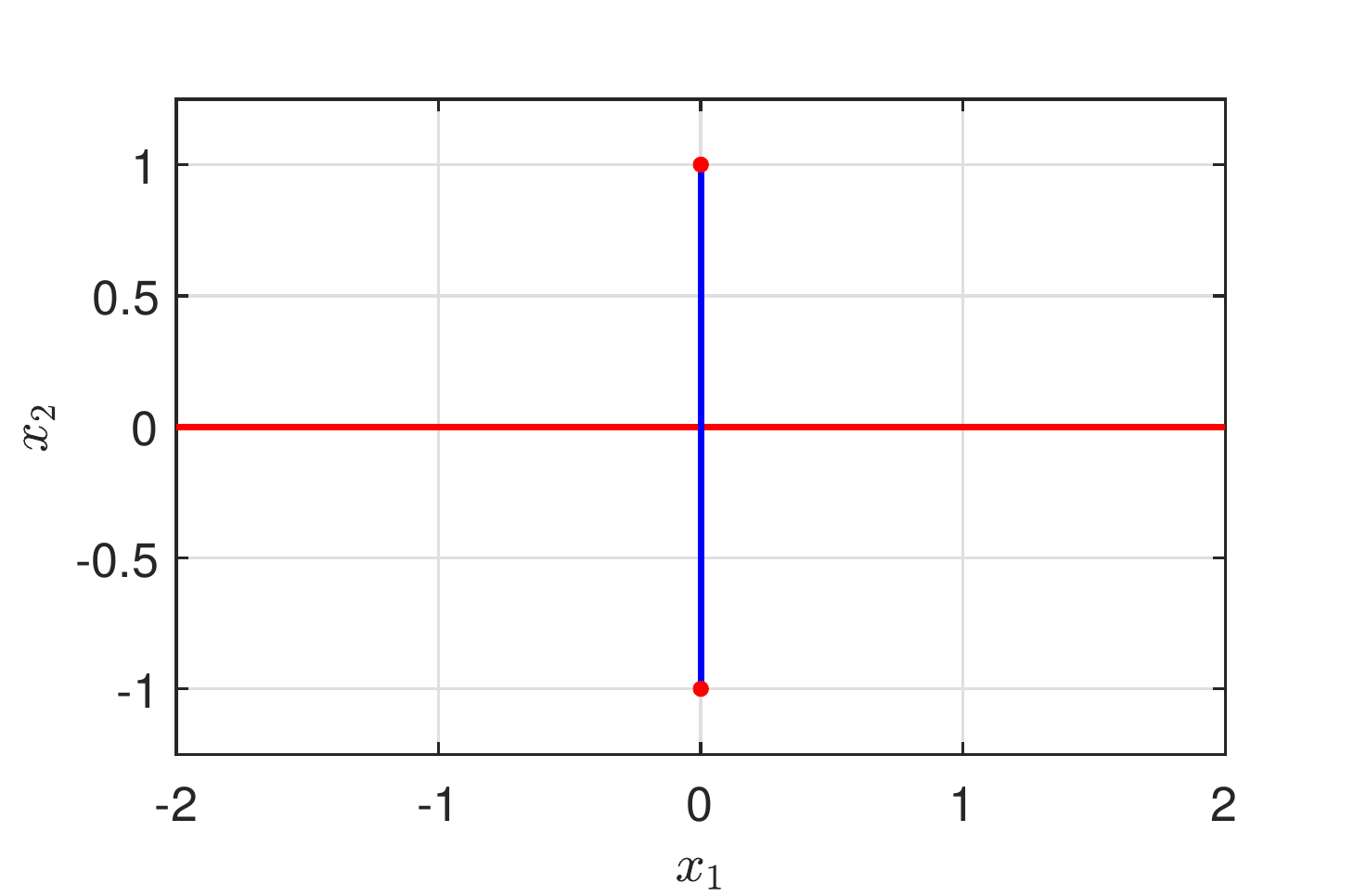}
		\caption{$P_{\mathsf{int}}$ (blue) and $P_0$ (red) for Example \ref{ex:isolTan}.}
		\label{fig:isolTan}
	\end{figure}
\end{example}

We will now take a closer look at the relationship between $Tan(P_{\mathsf{int}},x) = Tan(pr_x(\mathcal{M}),x)$ and $pr_x(T_p \mathcal{M})$ for $p = (x,\alpha) \in \mathcal{M}$. Note that for a general submanifold $\mathcal{M}$ of $\mathbb{R}^n$, we have $Tan(pr_x(\mathcal{M}),x) \neq pr_x(T_p \mathcal{M})$. For example, if $\mathcal{M} = S^1$ is the sphere in $\mathbb{R}^2$ and $pr_x$ is the projection onto the first coordinate (i.e. $pr_x(\mathcal{M}) = [-1,1]$), then for $\tilde{x} = (1,0)^{\top}$ the (one-dimensional) vector $-1$ is a tangent vector in $pr_x(\tilde{x}) = 1$, but $v_1 = 0$ for all tangent vectors $v = (v_1,v_2)^\top \in T_{\tilde{x}} \mathcal{M}$, so $pr_x(v) = 0$ for all $v \in T_{\tilde{x}} \mathcal{M}$. However, for the manifolds that arise in our context, we will show that, under some conditions, $Tan(P_{\mathsf{int}},x) = pr_x(T_p \mathcal{M})$.
The following lemma will show the first half of this statement, which is that the projection of a tangent vector of $\mathcal{M}$ is always in the tangent cone of $P_{\mathsf{int}}$ (without requiring any additional regularity assumption).

\begin{lemma} \label{lem:prTanSpace_in_tanCone}
	$pr_x(T_{(x_0,\alpha_0)} \mathcal{M}) \subseteq Tan(P_{\mathsf{int}},x_0) \quad \forall (x_0,\alpha_0) \in \mathcal{M}$.
\end{lemma}
\begin{proof}
	Let $p := (x_0,\alpha_0) \in \mathcal{M}$ and $v \in T_p \mathcal{M}$. If $pr_x(v) = 0$ we obviously have $pr_x(v) \in Tan(P_{\mathsf{int}},x_0)$. Let $pr_x(v) \neq 0$. By definition there exists a $C^1$-curve $\gamma: (-1,1) \rightarrow \mathcal{M}$ with $\gamma(0) = p$ and $\gamma'(0) = v$. Define $v_l := pr_x(\gamma(1/l)) - pr_x(p)$. W.l.o.g. assume that $v_l \neq 0$ $\forall l \in \mathbb{N}$. We have $v_l \rightarrow 0$ (since $\gamma$ is continuous) and $x_0 + v_l = pr_x(\gamma(1/l)) \in P_{\mathsf{int}}$. Finally, 
	\begin{align*}
		\frac{v_l}{\| v_l \|} &= \frac{pr_x(\gamma(1/l)) - pr_x(p)}{\| pr_x(\gamma(1/l)) - pr_x(p) \|} \\
		&= \frac{1/l}{\| pr_x(\gamma(1/l)) - pr_x(p) \|} \frac{pr_x(\gamma(1/l)) - pr_x(p)}{1/l} \\
		&=  \left| \left| pr_x \left( \frac{\gamma(1/l) - p}{1/l} \right) \right| \right|^{-1} pr_x \left( \frac{\gamma(1/l) - p}{1/l} \right) \rightarrow \frac{pr_x(v)}{\| pr_x(v) \|} \text{ as } l \rightarrow \infty,
	\end{align*}
	and hence, $pr_x(v) \in Tan(P_{\mathsf{int}},x_0)$.
\end{proof}

To show the opposite direction, i.e.~$pr_x(T_{(x_0,\alpha_0)} \mathcal{M}) \supseteq Tan(P_{\mathsf{int}},x_0)$, we first require an additional result about the uniqueness of KKT multipliers. 

\begin{lemma} \label{lem:alphaUnique}
	Let $x_0 \in P$. 
	\begin{enumerate}
		\item If $rk(Df(x_0)) = k - 1$ then $|A(x_0)| = 1$.
		\item If $x_0 \in P_{\mathsf{int}}$ and $|A(x_0)| = 1$ then $rk(Df(x_0)) = k - 1$.
		\item If $x_0 \in P_{\mathsf{int}}$ and $rk(Df(x_0)) < k-1$ then $|A(x_0)| > 1$ and $A(x_0) \cap \partial \Delta^{k-1} \neq \emptyset$. 
	\end{enumerate}
\end{lemma}
\begin{proof}

	1. Since $x_0 \in P$ we have $|A(x_0)| > 0$. Assume $|A(x_0)| > 1$ and let $\alpha^1, \alpha^2 \in A(x_0)$ with $\alpha^1 \neq \alpha^2$. Let $\tilde{\alpha}^1 := (\alpha^1, 1 - \sum_{i = 1}^{k-1} \alpha^1_i)^{\top} \in \R^k$ and $\tilde{\alpha}^2 := (\alpha^2, 1 - \sum_{i = 1}^{k-1} \alpha^2_i)^{\top} \in \R^k$. By definition we have $\tilde{\alpha}^1, \tilde{\alpha}^2 \in ker(Df(x_0)^\top)$ and $\sum_{i = 1}^k \tilde{\alpha}^1_i = \sum_{i = 1}^k \tilde{\alpha}^2_i = 1$. If there would be some $s \in \R$ with $\tilde{\alpha}^1 = s \tilde{\alpha}^2$ then this would result in
	\begin{equation*}
		1 = (1,...,1)^\top \tilde{\alpha}^1 = s (1,...,1)^\top \tilde{\alpha}^2 = s,
	\end{equation*}
	so $\tilde{\alpha}^1$ and $\tilde{\alpha}^2$ have to be linear independent. Thus $dim(ker(Df(x_0)^\top)) > 1$ which by Rank-nullity theorem means $rk(Df(x_0)) = rk(Df(x_0)^\top) = k - dim(ker(Df(x_0)^\top)) < k-1$.	
	
	2. Assume $rk(Df(x_0)) < k - 1$. Let $\alpha \in A(x_0)$. Then we must have some $w \in \mathbb{R}^k$ with $Df(x_0)^{\top} w = 0$ such that $(\alpha, 1 - \sum_{i = 1}^{k-1} \alpha_i)^{\top} \in \R^k$ and $w$ are linear independent. \\
	\emph{Case 1}: $\sum_{i = 1}^k w_i \neq 0$. Assume w.l.o.g.~$\sum_{i = 1}^k w_i = 1$, from which follows $\tilde{F}(x_0,(w_1,...,w_{k-1})^{\top}) = 0$. For $\lambda \in \mathbb{R}$ define $\beta(\lambda) := \alpha + \lambda ((w_1,...,w_{k-1})^{\top} - \alpha)$. Then we have $\tilde{F}(x_0,\beta(\lambda)) = 0$ for all $\lambda \in \mathbb{R}$. Since $\alpha \in \Delta^{k-1}$ and $\Delta^{k-1}$ is open there has to be some $\lambda \neq 0$ such that $\beta(\lambda) \in \Delta^{k-1}$, i.e.~$\alpha$ is not unique. \\
	\emph{Case 2}: $\sum_{i = 1}^k w_i = 0$. This means 
	\begin{align*}
	\sum_{i = 1}^k w_i \nabla f_i(x_0) = 0 &\Leftrightarrow \sum_{i = 1}^{k-1} w_i \nabla f_i(x_0) + (\sum_{i = 1}^{k-1} -w_i) \nabla f_k(x_0) = 0 \\
	& \Leftrightarrow \sum_{i = 1}^{k-1} w_i (\nabla f_i(x_0) - \nabla f_k(x_0)) = 0.
	\end{align*}
	Define $\beta(\lambda) := \alpha + \lambda (w_1,...,w_{k-1})^{\top}$. Then $\tilde{F}(x_0,\beta(\lambda)) = 0$ for all $\lambda \in \mathbb{R}$. The contradiction follows as in Case 1. 

	3. Let $\alpha \in A(x_0) \cap \Delta^{k-1}$. From the proof of 2.~we know that there is some $\gamma \in \mathbb{R}^{k-1}$ with $\tilde{F}(x_0,\alpha + \lambda \gamma) = 0$ for all $\lambda \in \mathbb{R}$. Since $\alpha \in \Delta^{k-1}$ and $\Delta^{k-1}$ is bounded, there has to be some $\lambda \in \mathbb{R}$ such that $\alpha + \lambda \gamma \in \partial \Delta^{k-1}$. In particular $|A(x_0)| > 1$.
\end{proof}
Lemma~\ref{lem:alphaUnique} has the following obvious implications:
\begin{corollary} \label{cor:rkTooSmall_rkFull}
	\begin{enumerate}
	\item Let $x_0 \in P$. If $rk(Df(x_0)) < k-1$ then $A(x_0) \cap \partial \Delta^{k-1} \neq \emptyset$.
	\item Let $x_0 \in P_{\mathsf{int}}$. Then $rk(Df(x_0)) = k-1$ $\Leftrightarrow$ $|A(x_0)| = 1$.
	\end{enumerate}
\end{corollary}
We can utilize the results about the uniqueness of the KKT multipliers to show that if the rank of $Df$ is large enough, the opposite direction of Lemma~\ref{lem:prTanSpace_in_tanCone} holds as well. For this we require $f$ to be three times continuously differentiable (by which $\tilde{F}$ is $C^2$) which we will assume from now on.
\begin{lemma} \label{lem:tanCone_in_prTanSpace}
	Let $x_0 \in P_{\mathsf{int}}$ with $rk(Df(x_0)) = k-1$. Then there exists $\alpha_0 \in A(x_0) \cap \Delta^{k-1}$ with 
	\begin{equation*}
		Tan(P_{\mathsf{int}},x_0) \subseteq pr_x(T_{(x_0,\alpha_0)} \mathcal{M}).
	\end{equation*}
\end{lemma}
\begin{proof}
	Let $v \in Tan(P_{\mathsf{int}},x_0)$ and $(v_i)_i$ be a sequence as in Definition \ref{def:tanCone}. Since $(x_0 + v_i)_i \subseteq P_{\mathsf{int}}$, this induces a sequence $(\alpha_i)_i \in \Delta^{k-1}$ with $\tilde{F}(x_0 + v_i, \alpha_i) = 0$ $\forall i \in \mathbb{N}$. Since $\Delta^{k-1}$ is bounded, we can w.l.o.g.~assume that $\alpha_i \rightarrow \alpha_0 \in \overline{\Delta^{k-1}}$ and by continuity of $\tilde{F}$ we have $\tilde{F}(x_0,\alpha_0) = 0$. By Corollary~\ref{cor:rkTooSmall_rkFull} we have $\alpha_0 \in \Delta^{k-1}$. \\
	Since $D_x \tilde{F}(x_0,\alpha_0)$ is invertible and $\tilde{F}$ is $C^2$, we can apply the Implicit Function Theorem at $(x_0,\alpha_0)$ to obtain neighborhoods $U \subseteq \Delta^{k-1}$ of $\alpha_0$, $V \subseteq \mathbb{R}^n$ of $x_0$ and a (unique) $C^2$ function $\phi : U \rightarrow V$ such that $\tilde{F}(x,\alpha) = 0 \Leftrightarrow \phi(\alpha) = x$ and $D\phi(\alpha) = -(D_x \tilde{F}(x,\alpha))^{-1} D_\alpha \tilde{F}(x,\alpha)$ $\forall (x,\alpha) \in V \times U$. Since $\alpha_i \rightarrow \alpha_0$ and $v_i \rightarrow 0$ we can assume w.l.o.g.~that $\alpha_i \in U$ and $x_0 + v_i \in V$ $\forall i \in \mathbb{N}$. Thus $\phi(\alpha_i) = x_0 + v_i$ and we get
	\begin{align*}
		\frac{v}{\| v \|} &= \lim_{i \rightarrow \infty} \frac{v_i}{\| v_i \|} = \lim_{i \rightarrow \infty} \frac{\phi(\alpha_i) - \phi(\alpha_0)}{\| v_i \|} = \lim_{i \rightarrow \infty} \frac{\phi(\alpha_0 + \| v_i \| \frac{\alpha_i - \alpha_0}{\| v_i \|}) - \phi(\alpha_0)}{\| v_i \|} \\
		&= \lim_{i \rightarrow \infty} \left( D\phi(\alpha_0) \frac{\alpha_i - \alpha_0}{\| v_i \|} + \frac{\mathcal{O}(\| \alpha_i - \alpha_0 \|^2)}{\| v_i \|} \right) \qquad \text{(Taylor)} \\
		&= \lim_{i \rightarrow \infty} \frac{\| \alpha_i - \alpha_0 \|}{\| v_i \|} \left( D\phi(\alpha_0) \frac{\alpha_i - \alpha_0}{\| \alpha_i - \alpha_0 \|} + \frac{\mathcal{O}(\| \alpha_i - \alpha_0 \|^2)}{\| \alpha_i - \alpha_0 \|^2} \| \alpha_i - \alpha_0 \| \right).
	\end{align*}
	Assume that $\left( \frac{\| \alpha_i - \alpha_0 \|}{\| v_i \|} \right)_i$ is unbounded. Then by the above equation we have $D\phi(\alpha_0) \frac{\alpha_i - \alpha_0}{\| \alpha_i - \alpha_0 \|} \rightarrow 0$. W.l.o.g. assume that $\frac{\alpha_i - \alpha_0}{\| \alpha_i - \alpha_0 \|} \rightarrow w$, which results in $D \phi(\alpha_0) w = 0$, thus $D_\alpha \tilde{F}(x_0,\alpha_0) w = 0$ with $\| w \| = 1$. This means
	\begin{align*}
		&(\nabla f_1(x_0) - \nabla f_k(x_0)) w_1 + \cdots + (\nabla f_{k-1}(x_0) - \nabla f_k(x_0)) w_{k-1} = 0 \\
		\Leftrightarrow \quad & w_1 \nabla f_1(x_0) + \cdots + w_{k-1} \nabla f_{k-1}(x_0) + \left( \sum_{i = 1}^{k-1} (-w_i) \right) \nabla f_k(x_0) = 0 \\
		\Leftrightarrow \quad & Df(x_0)^\top \left( w_1,...,w_{k-1},\sum_{i = 1}^{k-1} (-w_i) \right)^\top = 0 \\
		\Leftrightarrow \quad & \tilde{w} \in ker(Df(x_0)^\top)
	\end{align*}
	for $\tilde{w} := (w,\sum_{i = 1}^{k-1} (-w_i))^\top \in \R^k$. Let $\tilde{\alpha} := (\alpha, 1 - \sum_{i = 1}^{k-1} \alpha_i)^\top \in \R^k$. Then $\tilde{\alpha}$ and $\tilde{w}$ are linear independent since $\sum_{i = 1}^k \tilde{w}_i = 0$ and $\sum_{i = 1}^k \tilde{\alpha}_i = 1$. As they are both in $ker(Df(x_0)^\top)$ we must have $rk(Df(x_0)) < k-1$, which is a contradiction.
	
	So $\left( \frac{\| \alpha_i - \alpha_0 \|}{\| v_i \|} \right)_i$ has to be bounded and we can assume w.l.o.g.~that $\frac{\alpha_i - \alpha_0}{\| v_i \|} \rightarrow v^\alpha$. Thus
	\begin{align*}
		\frac{v}{\| v \|} &= \lim_{i \rightarrow \infty} \frac{\phi(\alpha_0 + \| v_i \| \frac{\alpha_i - \alpha_0}{\| v_i \|}) - \phi(\alpha_0)}{\| v_i \|} = D \phi(\alpha_0) v^\alpha \\
		&= -(D_x \tilde{F}(x_0,\alpha_0))^{-1} D_\alpha \tilde{F}(x_0,\alpha_0) v^\alpha.
	\end{align*}
	From this we obtain
	\begin{equation*}
		D \tilde{F}(x_0,\alpha_0) z = 0
	\end{equation*}
	with $z := \left(\frac{v}{\| v \|},v^\alpha \right)^\top \in \R^{n+k-1}$. By Theorem \ref{thm:Mmanifold} we have $T_{(x_0,\alpha_0)} \mathcal{M} = ker(D \tilde{F}(x_0,\alpha_0))$ which completes the proof.
\end{proof}

\subsection{The edge of the Pareto critical set}
We are now in the position to extend the results in \cite{P2017} to a more general situation and show that if $x$ lies on the ``edge'' of the Pareto critical set, then there has to be a corresponding KKT multiplier $\alpha$ with a zero component somewhere. The topological boundary $\partial P$ is in general not suitable for what we want to describe with the term ``edge'', see e.g. Example \ref{ex:isolTan} (where $\partial P = P$). Instead we define it in the following way:
\begin{definition} \label{def:P_E}
	We call 
	\begin{equation*}
	P_E := \{ x \in P : Tan(P_{\mathsf{int}},x) \neq -Tan(P_{\mathsf{int}},x) \}
	\end{equation*}
	the \emph{edge} of the Pareto critical set. In other words, $x_0 \in P_E$ iff there exists $v \in Tan(P_{\mathsf{int}},x_0)$ so that $-v \notin Tan(P_{\mathsf{int}},x_0)$.
\end{definition}


Building on Lemma~\ref{lem:prTanSpace_in_tanCone} and Lemma~\ref{lem:tanCone_in_prTanSpace}, we can now proof the main result of this section, namely that points on the edge of the Pareto critical set also satisfy the optimality conditions for a subproblem with at least one neglected objective.

\begin{theorem} \label{thm:main1}
	If $x_0 \in P_E$ then $A(x_0) \cap \partial \Delta^{k-1} \neq \emptyset$. 
\end{theorem}
\begin{proof}
	Assume the assertion does not hold, so $x_0 \in P_E$ and $\nexists \alpha \in \overline{\Delta^{k-1}} \setminus \Delta^{k-1}$ with $\tilde{F}(x_0,\alpha) = 0$. In particular, we get $x_0 \in P_{\mathsf{int}}$. Then -- according to Lemma~\ref{lem:alphaUnique} -- $\alpha$ has to be unique and contained in $\Delta^{k-1}$ and $rk(Df(x_0)) = k-1$. We can thus apply Lemma~\ref{lem:prTanSpace_in_tanCone} and Lemma~\ref{lem:tanCone_in_prTanSpace} to see that $Tan(P_{\mathsf{int}},x_0) = pr_x(T_{(x_0,\alpha_0)} \mathcal{M})$. Since $pr_x(T_{(x_0,\alpha_0)} \mathcal{M})$ is a vector space, we obviously have $Tan(P_{\mathsf{int}},x_0) = -Tan(P_{\mathsf{int}},x_0)$ which contradicts our assumption. 
\end{proof}

If we additionally assume that the rank of $Df$ is large enough we can use Lemma \ref{lem:alphaUnique} to get the following corollary:

\begin{corollary} \label{cor:main1cor}
	Let $rk(Df(x)) = k-1$ for all $x \in P_{\mathsf{int}}$. Then $P_E \subseteq P_0$.  
\end{corollary}

Theorem~\ref{thm:main1} and Corollary~\ref{cor:main1cor} show that some objective functions may be discarded if we are only interested in calculating the edge of the Pareto critical set. This will be used in Section~\ref{sec:CalculationParetoSetViaSubproblems}.

Although $\partial P$ and $P_0$ do not coincide in general, there is a special case where $\partial P = P_0$, as shown in the following lemma. (In this case, $P_E$ is not required to describe the relationship between $P_0$ and $P$).

\begin{lemma} \label{lem:PEP0}
	Let $rk(Df(x)) = k-1 = n$ for all $x \in P$. Then $\partial P = P_0$ (with $\partial P$ defined as in Definition \ref{def:topology}).
\end{lemma}
\begin{proof}
	By definition we have $\partial P = \overline{P} \setminus P^\circ$ and $P_0 = P \setminus P_{\mathsf{int}}$. To prove this assertion, we will show that $\overline{P} = P$ and $P^\circ = P_{\mathsf{int}}$. 
	
	$\overline{P} = P$: Let $(x_i)_i \in P$ be a sequence that converges to some $\bar{x} \in \mathbb{R}^n$ and $\alpha_i \in A(x_i)$. Since $\overline{\Delta^{k-1}}$ is compact we can assume w.l.o.g.~that $\alpha_i$ converges to some $\bar{\alpha} \in \overline{\Delta^{k-1}}$. Since $\tilde{F}$ is continuous, we have $0 = \lim_{i \rightarrow \infty} \tilde{F}(x_i,\alpha_i) = \tilde{F}(\bar{x},\bar{\alpha})$, and hence $\bar{x} \in P$. Consequently, $\overline{P} = P$ and $\partial P = P \setminus P^\circ$.
	
	$P^\circ = P_{\mathsf{int}}$: Let $x \in P_{\mathsf{int}}$. Assume that for all neighborhoods $U$ of $x$ in $\mathbb{R}^n$ there is some $y$ with $y \notin P_{\mathsf{int}}$. This means there is a sequence $(y_i)_i \in \mathbb{R}^n$ with $\lim_{i \rightarrow \infty} y_i = x$ and $y_i \notin P_{\mathsf{int}}$ for all $i \in \mathbb{N}$. By our assumption we have $Df(y) \in \mathbb{R}^{(n+1) \times n}$, so $dim(ker(Df(y)^{\top})) \geq 1$ for all $y \in \mathbb{R}^n$. Thus $(y_i)_i$ induces a sequence $(\beta_i)_i \in \mathbb{R}^{n+1} \setminus \{ 0 \}$ with 
	\begin{equation} \label{eq:betaKerDfT}
		Df(y_i)^{\top} \beta_i = 0 \Leftrightarrow \sum_{j = 1}^{n+1} (\beta_i)_j \nabla f_j(y_i) = 0 \quad \forall i \in \mathbb{N}.
	\end{equation}
	Since Equation~\eqref{eq:betaKerDfT} still holds if we scale every $\beta_i$ with some $\lambda_i \in \mathbb{R}^{>0}$, we can assume w.l.o.g.~that $\| \beta_i \| = 1$ for all $i \in \mathbb{N}$ and that it converges to some $\bar{\beta} \in ker(Df(x)^{\top})$ with $\| \bar{\beta} \| = 1$. Due to the assumption $rk(Df(x)) = n$, we have $dim(ker(Df(x)^{\top})) = 1$ such that $\alpha' := (\bar{\alpha},1 - \sum_{i = 1}^n \bar{\alpha}_i)$ and $\bar{\beta}$ have to be linear dependent. Since $\alpha' \in (\mathbb{R}^{>0})^n$, there is some $N \in \mathbb{N}$ and $\lambda \in \mathbb{R} \setminus \{ 0 \}$ such that $\lambda \beta_i \in (\mathbb{R}^{>0})^n$ for $i > N$ and Equation~\eqref{eq:betaKerDfT} still holds. Consequently, $(y_i)_i \in P_{\mathsf{int}}$ for $i > N$ which is a contradiction to our assumption. It follows that for all $x \in P_{\mathsf{int}}$, there is a neighborhood $U$ of $x$ in $\mathbb{R}^n$ with $U \subseteq P$, i.e.~$P_{\mathsf{int}} = P^\circ$. By this we obtain the desired result $\partial P = P \setminus P_{\mathsf{int}} = P_0$.
\end{proof}

\section{Calculating the Pareto critical set via lower-dimensional subproblems} \label{sec:CalculationParetoSetViaSubproblems}
In Section~\ref{sec:StructureParetoSet} we have shown that points on the edge $P_E$ of the Pareto critical set have a KKT multiplier where one component is zero.
We want to exploit this and consider the $k$ subproblems of \eqref{eq:MOP} where one objective function is neglected. By the results from Section~\ref{sec:StructureParetoSet}, $P_E$ is a subset of the union of the Pareto critical sets of these $k$ subproblems. (As we will see in Example~\ref{ex:reducedObj}, there are situations where this union contains a lot more than $P_E$.) Furthermore, we are going to study problems where more than one KKT multiplier is 0.


\subsection{Subproblems}
The subproblems mentioned above arise by omitting certain objective functions or, in other words, by only taking a subset of the set of objective functions: For $I \subseteq \{ 1,...,k \}$ we denote by
\begin{align}\label{eq:MOPI}
 	\min_{x \in \mathbb{R}^n} f^I(x) \tag{$\mbox{MOP}^I$}
\end{align}
the MOP where $f^I(x) := (f_i(x))_{i \in I}$, i.e. $f^I$ contains only those components of the objective function $f$ with indices in $I$. Let $P^I$ be the set of Pareto critical points of \eqref{eq:MOPI} and $F^I$, $\tilde{F}^I$ and $A^I$ be defined analogously to $F$, $\tilde{F}$ and $A$ in Section \ref{sec:StructureParetoSet} for \eqref{eq:MOPI}. Since we have not defined $P_0$ and $P_{\mathsf{int}}$ for scalar-valued MOPs, we set $P^I_0 := P^I$ and $P^I_{\mathsf{int}} := \emptyset$ if $|I| = 1$ for ease of notation. For $I = \emptyset$ let $P^I := \emptyset$ and $P_{\mathsf{int}} := \emptyset$. The following lemma shows that points that are Pareto critical with respect to a subset of the set of objective functions are also Pareto critical for the full problem: 

\begin{lemma} \label{lem:PIinP}
	$P^I \subseteq P$ $\forall I \subseteq \{1,...,k\}$.
\end{lemma}
\begin{proof}
	Let $I = \{ i_1,...,i_{|I|} \} \subseteq \{1,...,k\}$ and $x_0 \in P^I$. Then
	\begin{align*}
		\exists \alpha^I \in (\mathbb{R}^{\geq 0})^{|I|} : \nabla f_{i_1}(x_0) \alpha_1^I + \cdots + \nabla f_{i_{|I|}}(x_0) \alpha_{|I|}^I = 0 \text{ and } \sum_{i = 1}^{|I|} \alpha_i^I = 1.
	\end{align*}
	Define $\alpha \in \mathbb{R}^k$ via
	\begin{equation*}
		\alpha_i := \begin{cases}
			\alpha^I_i & i \in I, \\
			0 & \text{otherwise}.
		\end{cases}
	\end{equation*}
	Then $\sum_{i = 1}^k \alpha_i = \sum_{j = 1}^{|I|} \alpha^{|I|}_j = 1$,  
	\begin{equation*}
		\nabla f_1(x_0) \alpha_1 + \cdots + \nabla f_k(x_0) \alpha_k = 0,
	\end{equation*}
	and $x_0 \in P$ which yields the desired result.
\end{proof}

\begin{remark}
	Lemma~\ref{lem:PIinP} can be generalized to $P^J \subseteq P^K$ if $J \subseteq K$.
\end{remark}

In Corollary~\ref{cor:main1cor} we had to assume that $rk(Df(x)) = k-1$ for all $x \in P_{\mathsf{int}}$ to see that $P_E \subseteq P_0$. The following lemma shows that if this rank condition is violated, we can still find a subproblem \eqref{eq:MOPI} such that $x$ is Pareto critical with respect to \eqref{eq:MOPI} and $rk(Df^I(x)) = |I| - 1$, i.e.~that Corollary~\ref{cor:main1cor} can be applied to that subproblem. In particular, this induces a decomposition of \eqref{eq:MOP} into subproblems that satisfy the rank condition.

\begin{lemma} \label{lem:reducedObj}
	Let $x_0 \in P$. Then there exists some $I \subseteq \{1,...,k\}$ with $|I| = rk(Df(x_0)) + 1$ such that $rk(Df^I(x_0)) = rk(Df(x_0))$ and $x_0 \in P^I$. Moreover, if $x_0 \in P_0$ then $x_0 \in P^I_0$. In particular, there is some $J \subseteq \mathcal{P}(\{1,...,k\})$ with 
	\begin{equation*}
	P = \bigcup_{I \in J} P^I \quad \text{ and } \quad P_0 \subseteq \bigcup_{I \in J} P_0^I
	\end{equation*}
	and $\forall x \in P \exists I \in J$ with $rk(Df(x)) = rk(Df^I(x)) = |I| - 1$.
\end{lemma}
\begin{proof}
	If $rk(Df(x_0)) = k-1$ then we can simply choose $I = \{1,...,k\}$. We therefore now assume $rk(Df(x_0)) < k-1$.
	Let $J := \{ j \in \{1,...,k\} : rk(Df^{I \setminus \{j\} }(x_0)) < rk(Df(x_0)) \}$ be the set of linear independent objectives and $K := \{1,...,k\} \setminus J$. Then we have $\alpha_j = 0$ for all $j \in J$ and all $\alpha \in (\mathbb{R}^{\geq 0})^k$ with $F(x_0,\alpha) = 0$ (since the $j^{\mathsf{th}}$ gradient is linear independent). By construction we have $rk(Df^J(x_0)) = |J|$ and 
	\begin{align*}
		&k - 1 > rk(Df(x_0)) = rk(Df^K(x_0)) + |J| \\
		\Leftrightarrow \quad & rk(Df^K(x_0)) < k - |J| - 1 = |K| - 1.
	\end{align*}
	Consequently, $x_0 \in P^K$ and we can apply Corollary \ref{cor:rkTooSmall_rkFull} to $f^K$ to see that there is some $\alpha' \in (\mathbb{R}^{\geq 0})^{|K|}$ with $F^K(x_0,\alpha') = 0$ and some $l \in K$ with $\alpha'_l = 0$ and $rk(Df^{K \setminus \{l\} }(x_0)) = rk(Df^K(x_0))$ (since $K := \{1,...,k\} \setminus J$). 
	Thus, if we set $I = \{1,...,k\} \setminus \{l\}$ we have $rk(Df^I(x_0)) = rk(Df(x_0))$ and $x_0 \in P^I$. \\
	We will now show that if $x_0 \in P_0$, then after removing some $l$ as above, there is still a component where all KKT multipliers are zero. So assume $x_0 \in P_0$. Let $A'(x_0) := \{ \alpha \in (\mathbb{R}^{\geq 0})^k : \sum_{i = 1}^k \alpha_i = 1 \text{ and } F(x_0,\alpha) = 0 \}$. We will show that after neglecting the $l^{\mathsf{th}}$ objective, there still exists some $j \in I = \{1,...,k\} \setminus \{ l \}$ with $\alpha_j = 0$ for all $\alpha \in A'(x_0)$, so $x_0 \in P^I_0$. For all $\alpha \in A'(x_0)$ there is some $j \in \{1,...,k\}$ with $\alpha_j = 0$. By the structure of $A(x_0)'$, there has to be $j \in \{1,...,k\}$ with $\alpha_j = 0$ for all $\alpha \in A'(x_0)$. If there are two such indices, we are done since only one element of $I$ is removed. Hence, assume there is only one such $j$. \\
	We will show that $j \in J$ by contradiction, so assume $j \notin J$. Then $rk(Df(x_0)) = rk(Df^{I \setminus \{j\}}(x_0))$, and there is some $\beta \in \mathbb{R}^k$ with $\beta_j = 0$ and $Df(x_0)^{\top} \beta = \nabla f_j(x_0)$. If we set $\beta' := -\beta + e_j$, where $e_j$ is the $j$-th unit vector (in Cartesian coordinates), then $F(x_0,\beta') = 0$. By the assumption of the uniqueness of $j$ there has to be some $\alpha' \in A'(x_0)$ with $\alpha'_i > 0$ for all $i \neq j$. So there has to be $s > 0$ such that $\gamma \in A'(x_0)$ where
	\begin{equation*}
		\gamma := \frac{\alpha' + s \beta'}{1 - s \sum_{i = 1}^{k} \beta'_i}.
	\end{equation*}
	But we have $\gamma_j = \frac{s}{1 - s \sum_{i = 1}^{k} \beta'_i} \neq 0$ which is a contradiction to $\alpha_j = 0$ for all $\alpha \in A'(x_0)$, hence $j \in J$. As a consequence, in the above step $j$ can not be removed since $l \notin J$. \\
	Each time we apply the above procedure, $|I|$ is decreased by $1$ and $rk(Df^I(x_0))$ does not change, so there has to be some $I$ with $|I| - 1 = rk(Df(x_0))$ and $x_0 \in P^I$. Also $x_0 \in P^I_0$ if $x_0 \in P_0$. \\
\end{proof}

Lemma~\ref{lem:reducedObj} also shows that it suffices to solve a number of subproblems with fewer objective functions instead of the full MOP to get the complete Pareto critical set (and not only the edge) if the rank of the Jacobian of $f$ is small (relative to $k$). (For instance, such a situation always occurs when $k > n + 1$.) This is especially useful since the complexity for solving MOPs in general increases significantly with the number of objectives. Unfortunately, Lemma \ref{lem:reducedObj} does not help for finding the smallest set $J$ of subsets of the objectives that need to be solved, so in general we have to consider all possible subsets. An example for Lemma~\ref{lem:reducedObj} is shown below:

\begin{example} \label{ex:reducedObj}
	a) Consider the MOP $\min_{x \in \mathbb{R}^2} f(x)$ with
	\begin{equation*}
		f(x) := \begin{pmatrix}
			(x_1 + 1)^2 + (x_2 + 1)^2 \\
			(x_1 - 1)^2 + (x_2 + 1)^2 \\
			x_1^2 + (x_2 - 1)^2 \\
			x_1^2 + x_2^2
		\end{pmatrix}.
	\end{equation*}
	The Pareto critical set is the triangle with corners $(-1,-1)$, $(1,-1)$ and $(0,1)$. Since $k = 4$ and $n = 2$, we must have $rk(Df(x)) < k - 1 = 3$. By Lemma~\ref{lem:reducedObj} we can write $P$ as the union of some $P^I$. We can choose for example to solve \eqref{eq:MOPI} with $I \in \{ \{1,2,4\}, \{2,3,4\}, \{1,3,4\} \}$ or with $I = \{ 1,2,3 \}$. The situation is depicted in Figure \ref{fig:reducedObjA}. (Note that in general, it will obviously not be sufficient to only solve one \eqref{eq:MOPI} as will be seen in the next example.) \\
	b) Consider the MOP $\min_{x \in \mathbb{R}^2} f(x)$ with
	\begin{equation*}
		f(x) := \begin{pmatrix}
			(x_1 + 1)^2 + (x_2 + 1)^2 \\
			(x_1 - 1)^2 + (x_2 + 1)^2 \\
			(x_1 - 1)^2 + (x_2 - 1)^2 \\
			(x_1 + 1)^2 + (x_2 - 1)^2
		\end{pmatrix}.
	\end{equation*}
	The Pareto critical set is the square with corners $(-1,-1)$, $(1,-1)$, $(1,1)$ and $(-1,1)$. Here we can choose for example $I \in \{ \{1,2,4\}, \{2,3,4\} \}$ (the bottom-left triangle and the upper-right triangle), which is depicted in Figure \ref{fig:reducedObjB}. Note that Lemma \ref{lem:reducedObj} does not state anything about $P^I_{\mathsf{int}}$, so if we apply it in some $x_0 \in P$ we generally do not know if there is some $I$ (that satisfies the rank condition) with $x_0 \in P^I_{\mathsf{int}}$. In this example there is no $I$ so that $(0,0)^{\top} \in P_{\mathsf{int}}^I$ although $(0,0)^{\top} \in P_{\mathsf{int}}$.
	\begin{figure}[ht] 
		\begin{subfigure}[t]{.5\textwidth}
			\centering
			\includegraphics[scale=0.6]{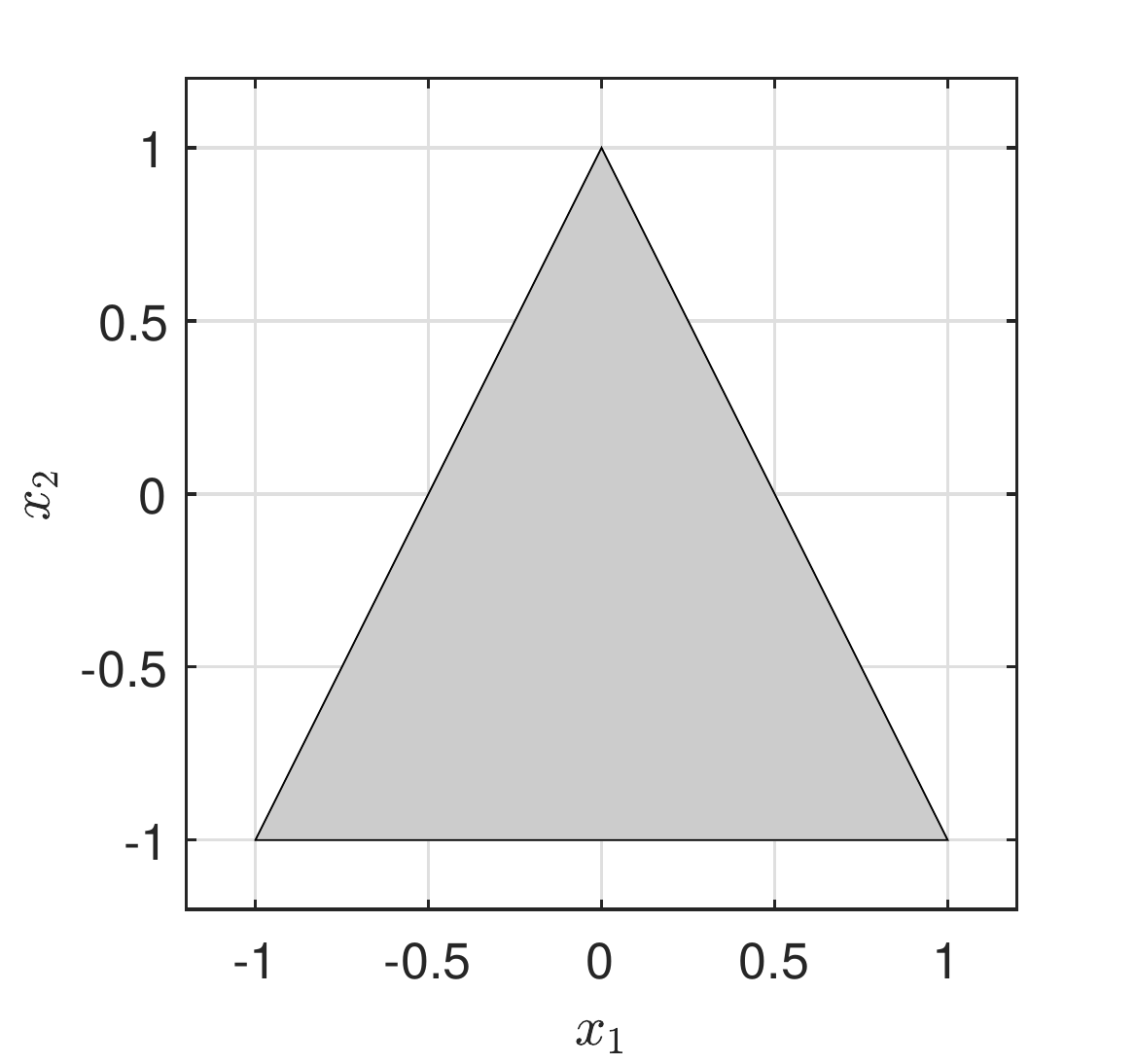}
			\caption{$P$}
		\end{subfigure}
		\begin{subfigure}[t]{.49\textwidth}
			\centering
			\includegraphics[scale=0.6]{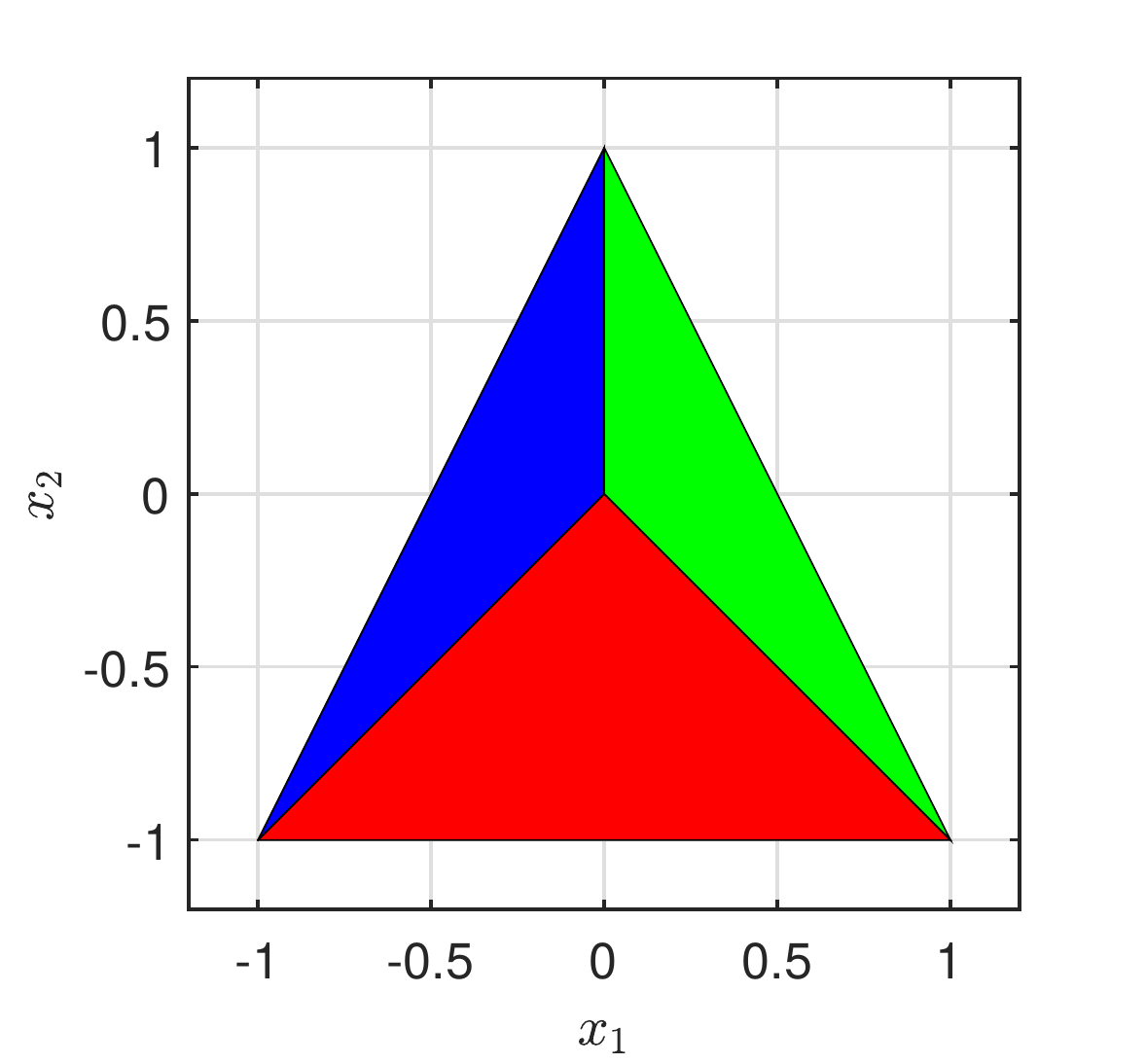}
			\caption{A decomposition of $P$}
		\end{subfigure}
		\caption{$P$ and the decomposition $\{ \{1,2,4\} \text{ (red)}, \{2,3,4\} \text{ (green)}, \{1,3,4\} \text{ (blue)} \}$ for Example \ref{ex:MNoManifold}, a).}
		\label{fig:reducedObjA}
	\end{figure}
	
	\begin{figure}[ht] 
		\begin{subfigure}[t]{.5\textwidth}
			\centering
			\includegraphics[scale=0.6]{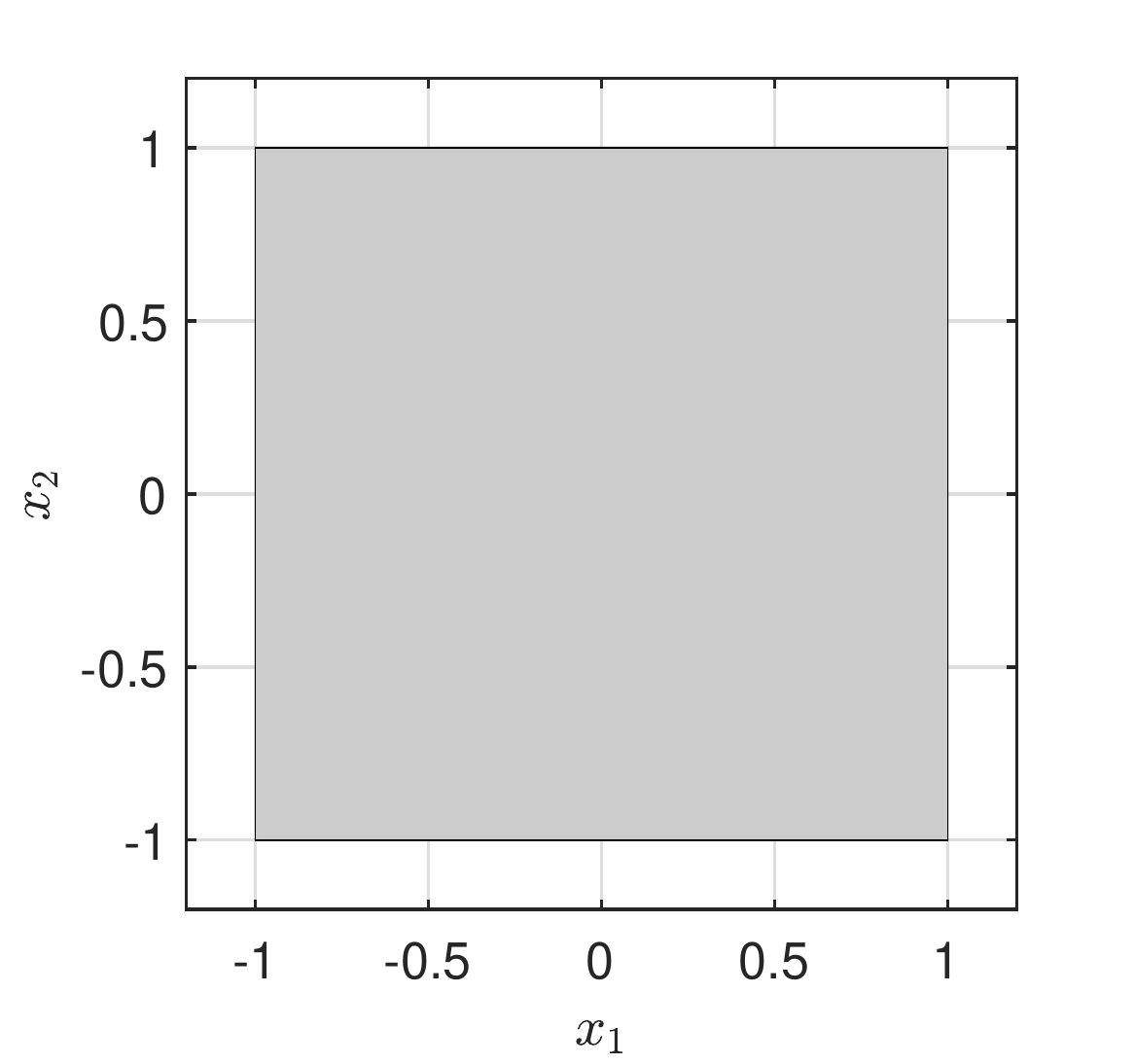}
			\caption{$P$}
		\end{subfigure}
		\begin{subfigure}[t]{.49\textwidth}
			\centering
			\includegraphics[scale=0.6]{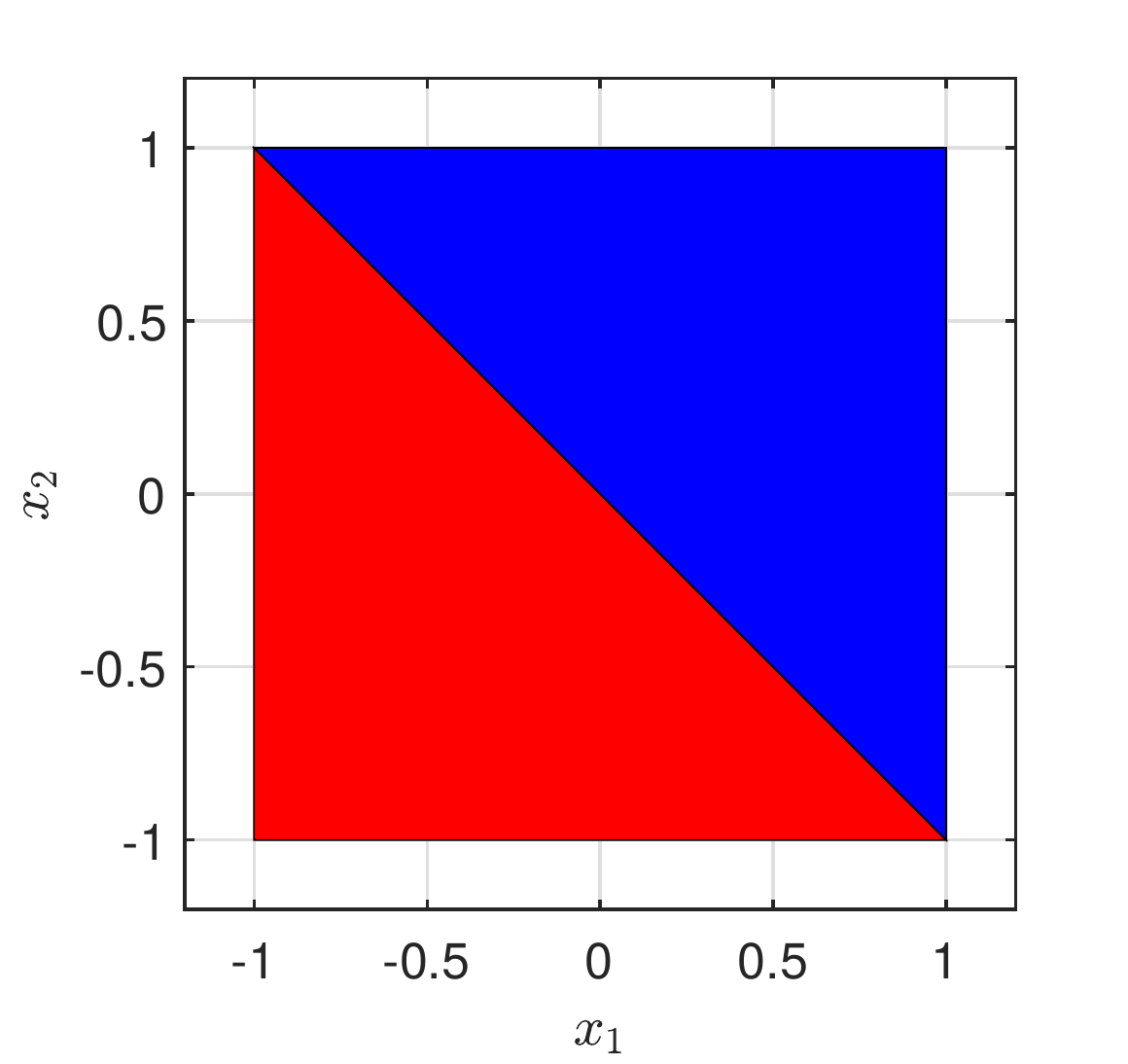}
			\caption{A decomposition of $P$}
		\end{subfigure}
		\caption{$P$ and the decomposition $\{ \{1,2,4\} \text{ (red)}, \{2,3,4\} \text{ (blue)} \}$ for Example \ref{ex:MNoManifold}, b).}
		\label{fig:reducedObjB}
	\end{figure}
\end{example}

As described above, Lemma \ref{lem:reducedObj} states how the complete Pareto critical set can be obtained by solving subproblems with lesser objective functions. Roughly speaking, these subproblems are made so that we can apply Theorem \ref{thm:main1} to see that one additional objective can be omitted to only get the edge of the critical set. This is done in the following lemma.

\begin{lemma} \label{lem:edgeInJ}
	Let $x_0 \in P_E$ and $m := \max_{x \in P} rk(Df(x))$. Then there has to be some $I \in \mathcal{P}(\{1,...,k\})$ with $|I| \leq m + 1$ such that either $I = \{ i \}$ and $\nabla f_i(x_0) = 0$ or $A^I(x_0) \cap \partial \Delta^{|I|-1} \neq \emptyset$.
\end{lemma}
\begin{proof}
	Let $J$ be as in Lemma \ref{lem:reducedObj} and $x_0 \in P_E$. By construction we have $|I| \leq m+1$ for all $I \in J$. If there is some $I \in J$ with $x_0 \in P^I$ and $|I| = 1$ we have $\nabla f_i(x_0) = 0$ for $I = \{ i \}$ and we are done. So assume $|I| > 1$ for all $I \in J$ with $x_0 \in P^I$. \\
	\emph{Case 1}: $Tan(P_{\mathsf{int}},x_0) \neq Tan(P,x_0)$. We must have $v \in Tan(P,x_0)$ with $v \notin Tan(P_{\mathsf{int}},x_0)$. Since $P_0 = P \setminus P_{\mathsf{int}}$, this means $U \cap P_0 \neq \emptyset$ for all neighborhoods $U$ of $x_0$ and thus, $x_0 \in \overline{P_0}$. Let $(y_j)_j \in P_0$ with $\lim_{j \rightarrow \infty} y_j = x_0$. By Lemma \ref{lem:reducedObj} and since $|J|$ is finite, there has to be some $K \in J$ with $y_j \in P_0^K$ infinitely many times such that $x_0 \in \overline{P_0^K}$. In particular, $\partial \Delta^{|K|-1} \cap A^K(x_0) \neq \emptyset$. \\
	\emph{Case 2}: $Tan(P_{\mathsf{int}},x_0) = Tan(P,x_0)$. This means
	\begin{equation} \label{eq:TanUnion}
		Tan(P_{\mathsf{int}},x_0) = Tan(P,x_0) = Tan \left( \bigcup_{I \in J} P^I, x_0 \right) = \bigcup_{I \in J} Tan(P^I,x_0).
	\end{equation} 
	For the last equality note that $\supseteq$ is obvious and $\subseteq$ follows from the fact that $P^I$ is closed for all $I$. Since $x_0 \in P_E$ there has to be $I \in J$ with $Tan(P^I,x_0) \neq -Tan(P^I,x_0)$. If $Tan(P^I_{\mathsf{int}},x_0) \neq Tan(P^I,x_0)$, we have $x_0 \in \overline{P_0^I}$ and we are done as in Case 1. Otherwise, we get $x_0 \in P_E^I$. If $rk(Df^I(x_0)) = |I| - 1$ we can apply Theorem \ref{thm:main1} to obtain $x_0 \in P_0^I$ (so $A^I(x_0) \subseteq \partial \Delta^{|I|-1}$). If $rk(Df^I(x_0)) < |I| - 1$ we can apply Corollary \ref{cor:rkTooSmall_rkFull} to \eqref{eq:MOPI} which yields $\partial \Delta^{|I|-1} \cap A^I(x_0) \neq \emptyset$. 
	
\end{proof}

By uniting over all possible subsets of $\{1,...,k\}$ of appropriate size, the last lemma can be used to get the following corollary.

\begin{corollary} \label{cor:edgeInReduced}
	Assume $m := \max_{x \in P} rk(Df(x)) > 0$. Then
	\begin{equation*}
		P_E \subseteq \bigcup_{I \in \mathcal{P}(\{1,...,k\}), |I| = m} P^I.
	\end{equation*}
\end{corollary}

Corollary \ref{cor:edgeInReduced} is the main result of this section. It basically states that it suffices to consider $|I| = rk(Df)$ objective functions to compute the edge of the Pareto critical set. We will demonstrate this in the following examples.

\subsection{Examples}

\begin{example} \label{ex:reducedObjExt}
	We again consider the MOPs from Example~\ref{ex:reducedObj}. \\
	a) We have $rk(Df(x)) = 2$ for all $x \in \mathbb{R}^2$, so by Corollary \ref{cor:edgeInReduced} it suffices to consider only pairs of $2$ objective functions. The corresponding Pareto critical sets are shown in Figure~\ref{fig:reducedObjExt}~(a). To compute $P_E$ it would suffice to consider the three subproblems in $\{ \{1,2\}, \{2,3\}, \{1,3\} \}$. \\
	b) We have $rk(Df(x)) = 2$ for all $x \in \mathbb{R}^2$, so again it suffices to consider only pairs of $2$ objective functions. In this case we need to solve the four subproblems in $\{ \{1,2\}, \{2,3\}, \{3,4\}, \{4,1\} \}$ to obtain $P_E$ as shown in Figure \ref{fig:reducedObjExt} (b).
	\begin{figure}[ht] 
		\begin{subfigure}[t]{.5\textwidth}
			\centering
			\includegraphics[scale=0.6]{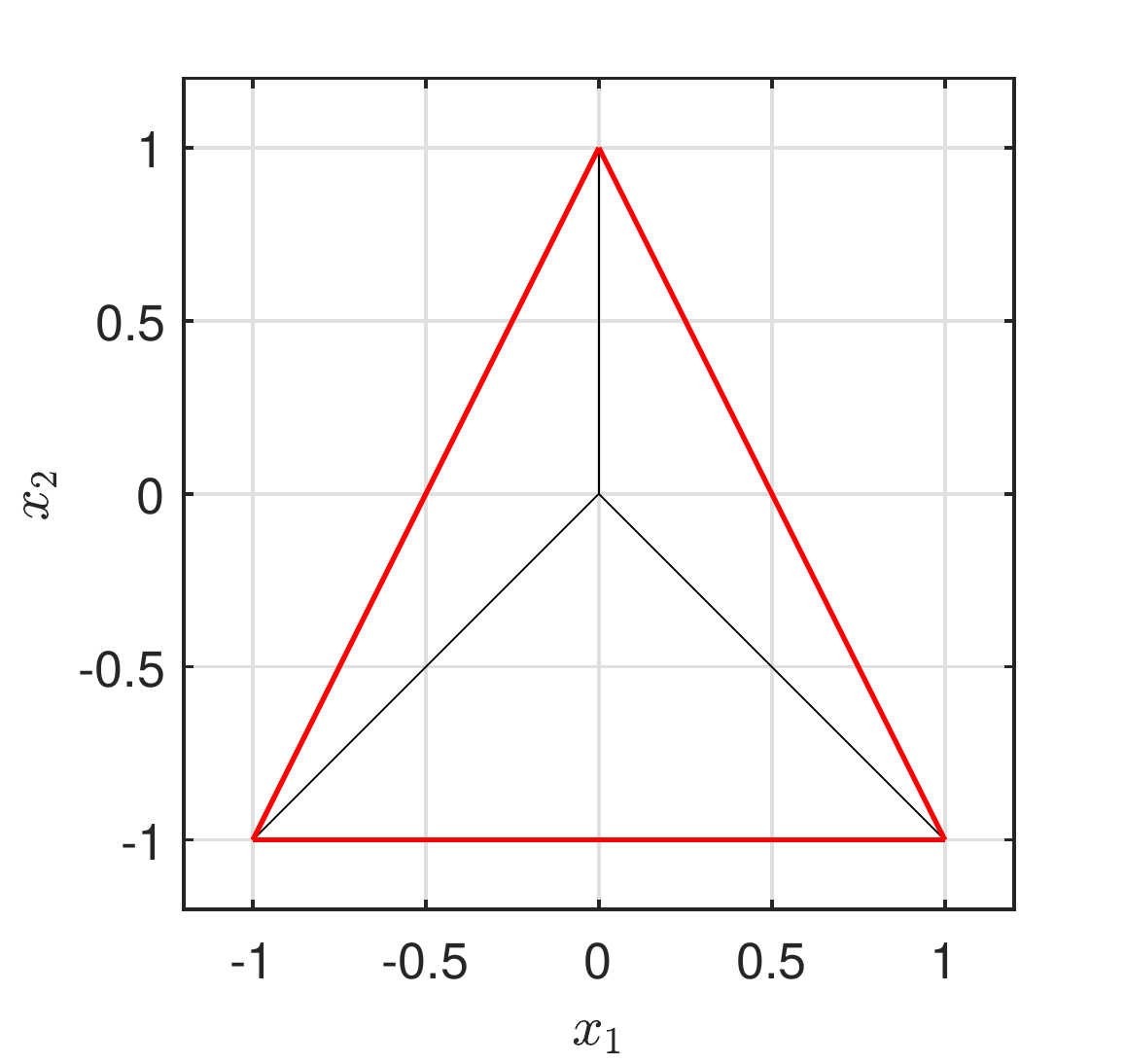}
			\caption{Example \ref{ex:reducedObj}, a)}
		\end{subfigure}
		\begin{subfigure}[t]{.49\textwidth}
			\centering
			\includegraphics[scale=0.6]{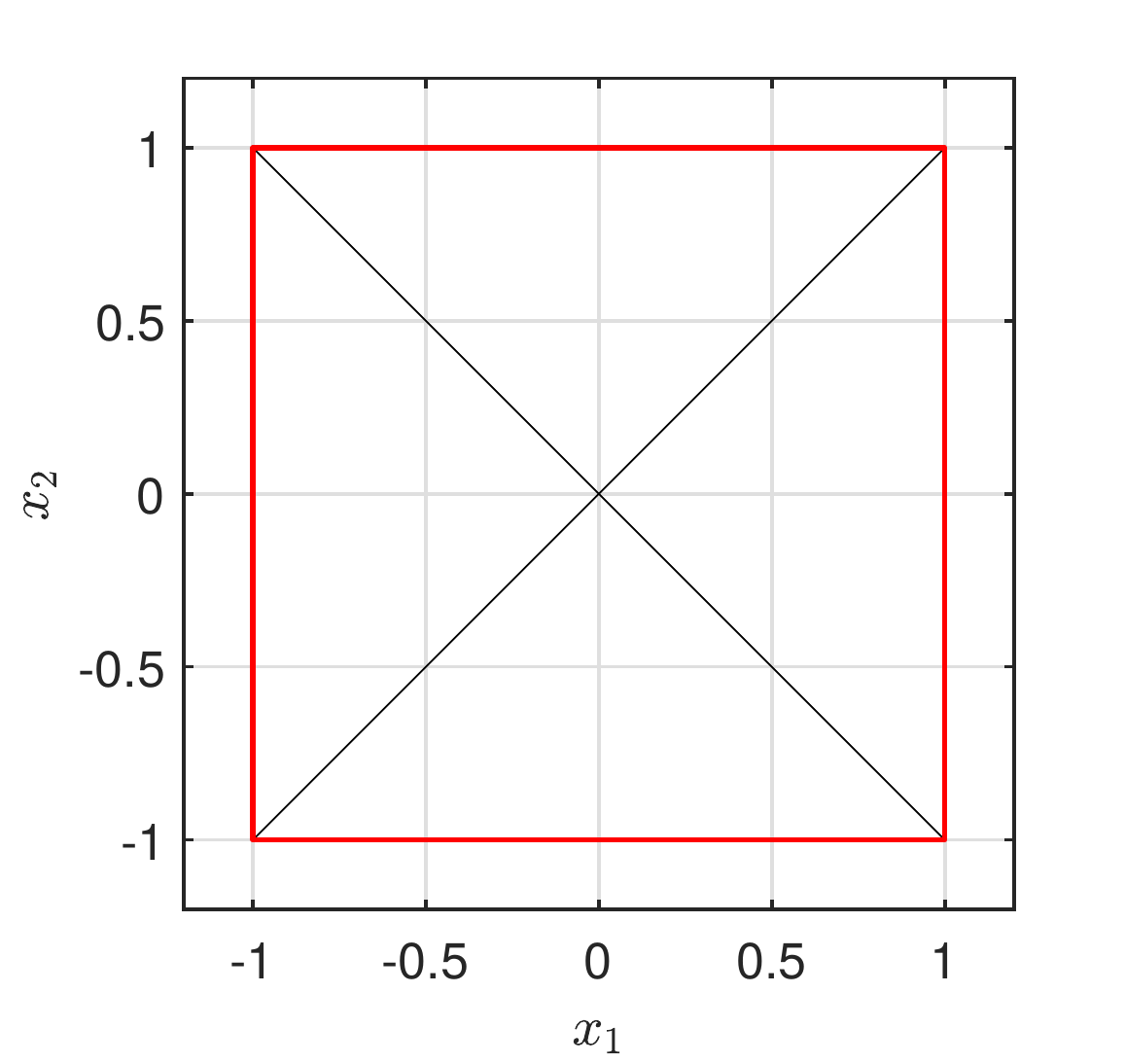}
			\caption{Example \ref{ex:reducedObj}, b)}
		\end{subfigure}
		\caption{Necessary (red) and unnecessary (black) part of Pareto critical set of the subproblems.}
		\label{fig:reducedObjExt}
	\end{figure}
\end{example}

The above examples are obviously very simple such that the relation between $P$, $P_E$, $P_{\mathsf{int}}$ and $P_0$ is relatively easy to see. We will now consider more complicated examples:

\begin{example} \label{ex:complicatedReg}
	Consider the MOP $\min_{x \in \mathbb{R}^2} f(x)$ with
	\begin{equation*}
		f(x) := \begin{pmatrix}
			x_1^4 + x_2^4 \\
			(x_1 - 1/3)^6 + (x_2 - 1/3)^2 \\
			(x_1 - 2/3)^2 + (x_2 - 2/3)^4 \\
			0.25 (x_1 - 1)^2 + (x_2 - 1)^4
		\end{pmatrix}.
	\end{equation*}
	By construction the hessian matrices of the individual objectives are diagonal and it is easy to see that $D_x \tilde{F}(x,\alpha)$ is invertible for all $(x,\alpha) \in \mathbb{R}^2 \times \Delta^{3}$, so the assumption (\ref{eq:assum_DxF_reg}) is satisfied. The Pareto critical set is shown in Figure \ref{fig:complReg_Pareto} (a), where the black dots indicate the critical points of each objective individually. Figure \ref{fig:complReg_Pareto} (b) shows the solutions to all possible 2-objective subproblems. Figure \ref{fig:complReg_colors} shows which Pareto critical set corresponds to which 2-objective subproblem. \\
	In contrast to Example \ref{ex:reducedObjExt} we see that it is possible for the Pareto critical set of a subproblem to be partly on $P_E$ and partly inside of $P$, for example in the case of objective $1$ and $3$ (green). \\
	Additionally, we see intersections of critical sets outside of the solution of the 1-objective subproblems, for example the intersection of the red and the blue line. This indicates that these points have two KKT multipliers with different zero components.
	\begin{figure}[ht] 
		\begin{subfigure}[t]{.5\textwidth}
			\centering
			\includegraphics[scale=0.5]{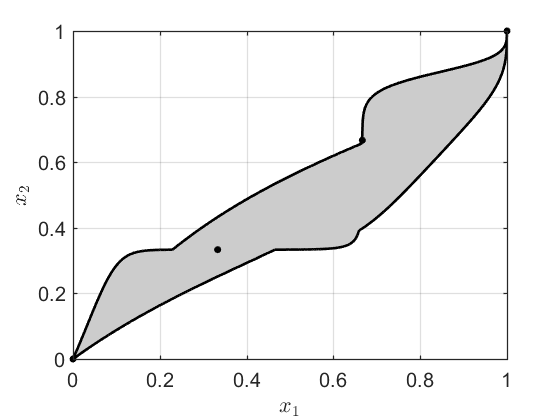}
			\caption{$P$ (gray) and $P_E$ (black)}
		\end{subfigure}
		\begin{subfigure}[t]{.49\textwidth}
			\centering
			\includegraphics[scale=0.5]{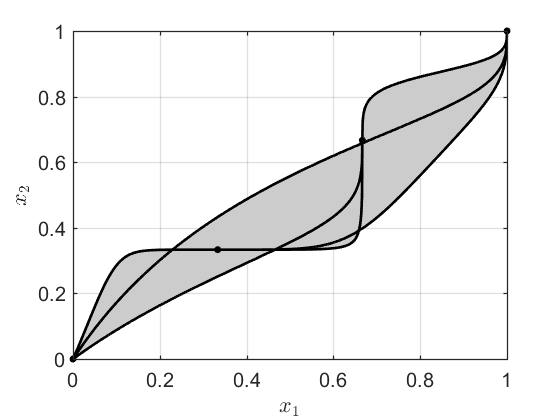}
			\caption{$P$ (gray) and the solutions of the subproblems (black)}
		\end{subfigure}
		\caption{Pareto critical sets for Example \ref{ex:complicatedReg}.}
		\label{fig:complReg_Pareto}
	\end{figure}
	\begin{figure}[t]
		\centering
		\includegraphics[scale=0.7]{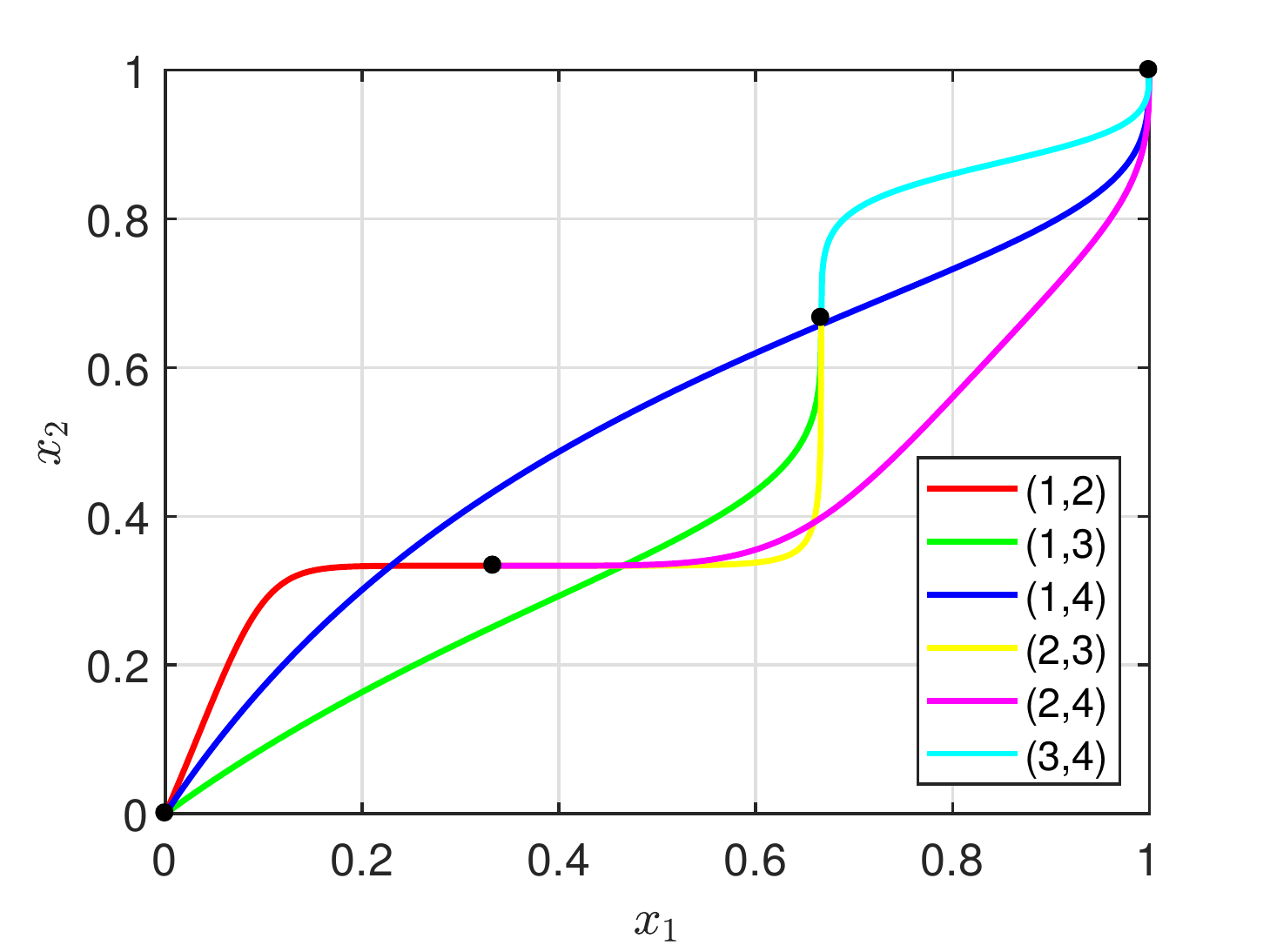}
		\caption{Pareto critical set for each 2-objective subproblem}
		\label{fig:complReg_colors}
	\end{figure}
\end{example}

The previous example indicates that if we have a ``kink'' in $P_E$ then it is either a critical point of a $(m-1)$-objective subproblem (with $m$ as in Corollary \ref{cor:main1cor}) or a critical point with multiple KKT multipliers on the boundary of the standard simplex. The classification of those non-differentiabilities in the boundary of the Pareto critical set highlights an additional advantage of the approach presented in this paper.

As shown in Lemma \ref{lem:localManifold}, we can remove points that do not satisfy the assumption (\ref{eq:assum_DxF_reg}) from the extended Pareto critical set $\mathcal{M}$ to still have a manifold structure. Since the techniques we used in Section \ref{sec:StructureParetoSet} and \ref{sec:CalculationParetoSetViaSubproblems} were basically of local nature, this encourages that our results can also be applied to MOPs that do not satisfy assumption (\ref{eq:assum_DxF_reg}). This will be done in the following examples.

\begin{example} \label{ex:complicatedIrreg}
	Consider the MOP $\min_{x \in \mathbb{R}^2} f(x)$ with
	\begin{equation*}
		f(x) := \begin{pmatrix}
			0.5 (x_1 - 1)^2 + x_2^2 \\
			2 x_1^2 + 2 (x_2 - 1)^2 \\
			2 (x_1 + 1)^2 + x_2^5 \\
			-2 x_1^3 + 2(x_2 + 1)^2
		\end{pmatrix}.
	\end{equation*}
	Since each objective function is polynomial, it is still (relatively) easy to calculate the Pareto critical set analytically. The part of interest is shown in Figure~\ref{fig:complIrreg_Pareto}~(a). Figure~\ref{fig:complIrreg_Pareto}~(b) additionally shows the Pareto critical sets to all 2-objective subproblems. Figure~\ref{fig:complIrreg_colors} shows which Pareto critical set corrseponds to which 2-objective subproblem. \\
	One can see that for this example, it is not necessary to consider the subproblem $\{1,4\}$ since its Pareto critical set is in the interior -- i.e.~it does not lie on the edge -- of the actual Pareto critical set. (Strictly speaking $P_E \cap P^{\{1,4\}} = \{ (0,-1) \}$, but this point is also in $P^{\{2,4\}}$ and $P^{\{3,4\}}$, so it is already covered). Other than that, all subproblems have to be solved to obtain $P_E$.
	\begin{figure}[ht] 
		\begin{subfigure}[t]{.5\textwidth}
			\centering
			\includegraphics[scale=0.525]{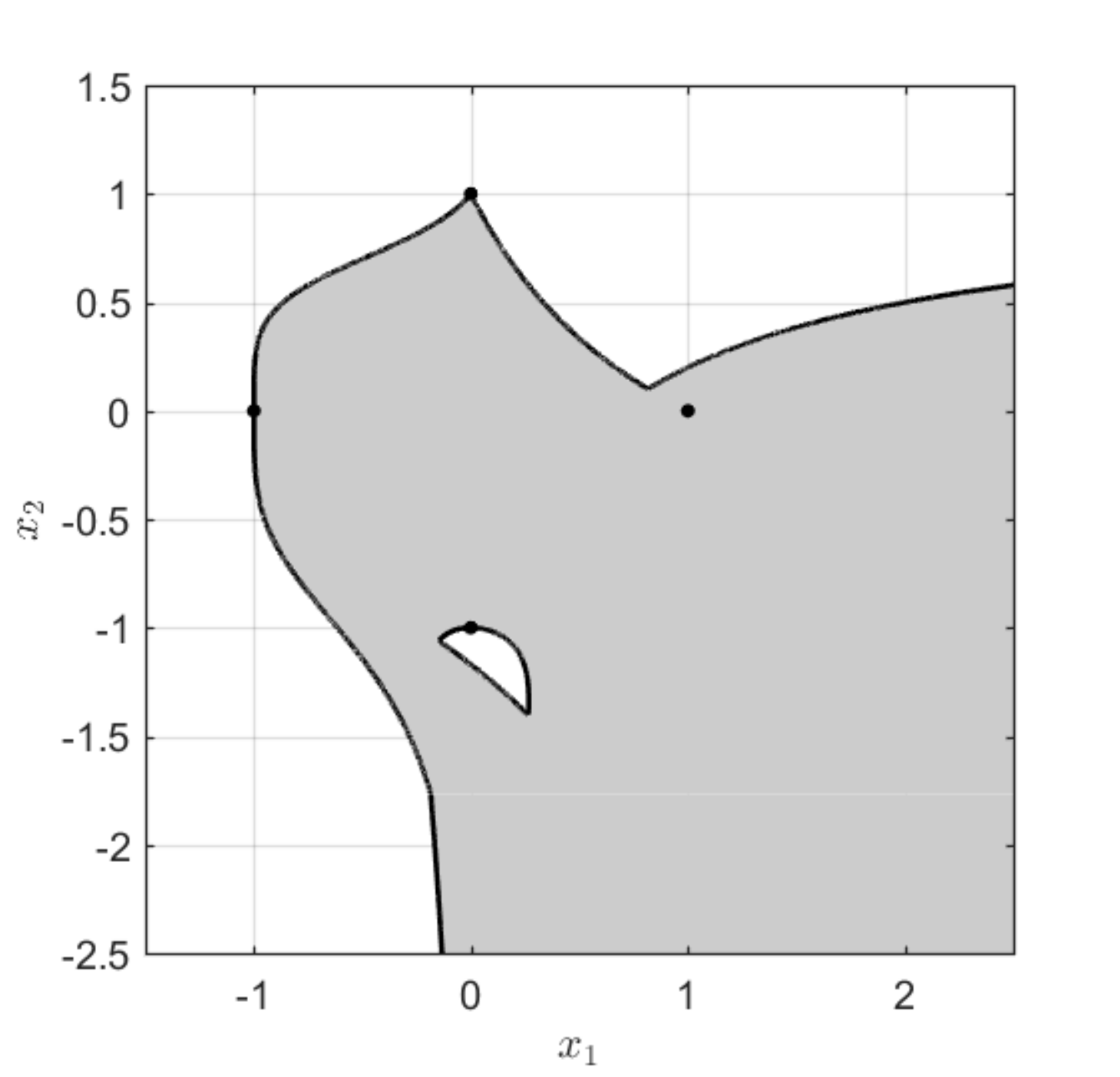}
			\caption{$P$ (gray) and $P_E$ (black)}
		\end{subfigure}
		\begin{subfigure}[t]{.49\textwidth}
			\centering
			\includegraphics[scale=0.525]{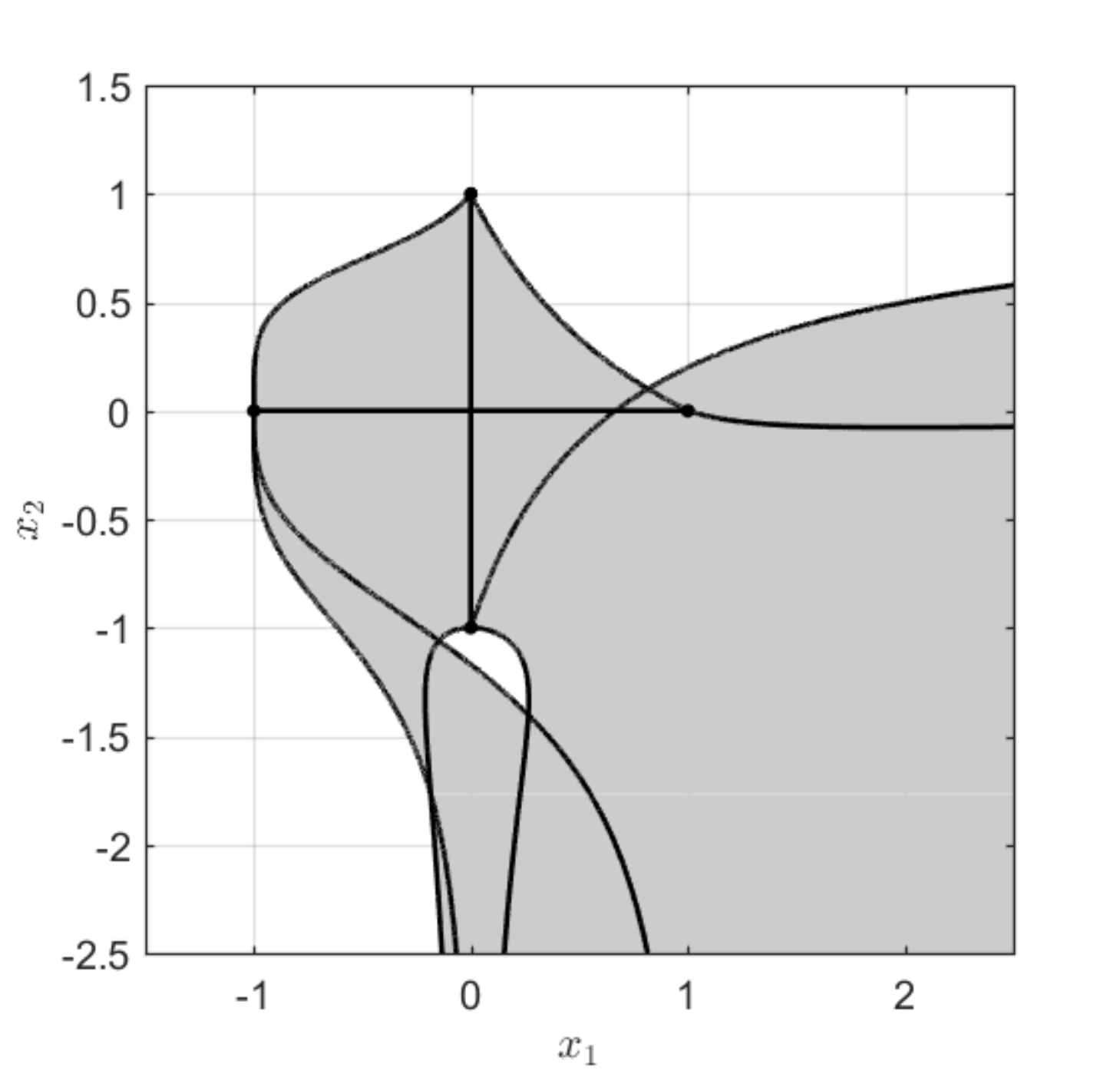}
			\caption{$P$ (gray) and the solutions of the subproblems (black)}
		\end{subfigure}
		\caption{Pareto critical sets for Example \ref{ex:complicatedIrreg}.}
		\label{fig:complIrreg_Pareto}
	\end{figure}
	\begin{figure}[t]
		\centering
		\includegraphics[scale=0.7]{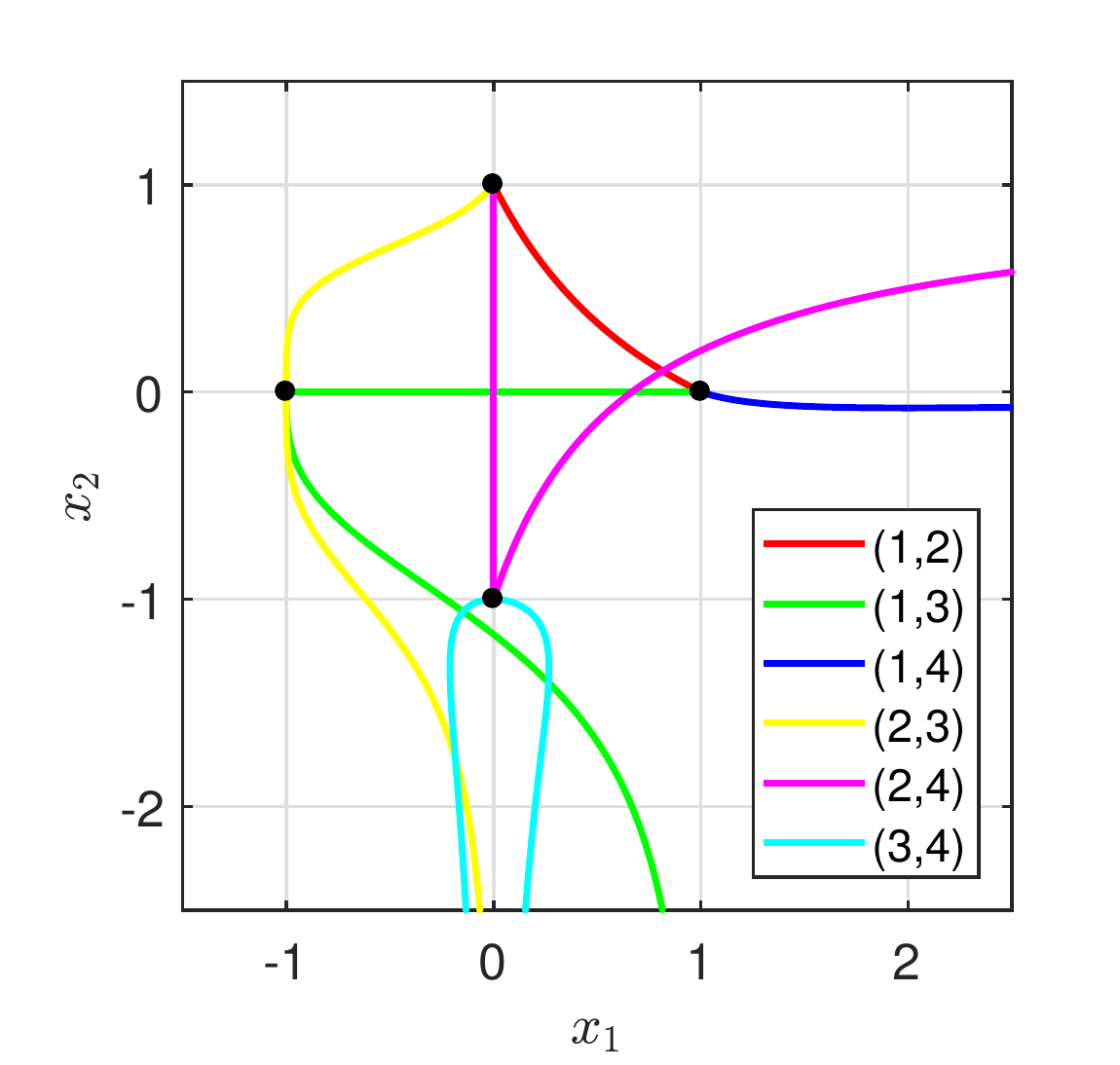}
		\caption{Pareto critical set for each 2-objective subproblem in Example \ref{ex:complicatedIrreg}.}
		\label{fig:complIrreg_colors}
	\end{figure}
\end{example}

The following example from \cite[Example 4.1.5]{P2017} shows how the Pareto critical set can be derived from $P_E$ if additional properties of the objective function are known, like in this case boundedness.

\begin{example} \label{ex:disc}
	Consider the MOP $min_{x \in \mathbb{R}^2} f(x)$ with
	\begin{equation*}
		f(x) := \begin{pmatrix}
			-6 x_1^2 + x_1^4 + 3 x_2^2 \\
			(x_1 - 0.5)^2 + 2 (x_2 - 1)^2 \\
			(x_1 - 1)^2 + 2 (x_2 - 0.5)^2
		\end{pmatrix}.
	\end{equation*}
	Figure~\ref{fig:disc_color} shows the Pareto critical sets of the three 2-objective subproblems. It is possible to show that $rk(Df(x)) = 2$ for all $x \in P_{\mathsf{int}}$ so that Lemma \ref{lem:PEP0} can be applied. The Pareto critical set of this problem is bounded so we know that it is given by the interior (and boundary) of the two disconnected sets depicted in Figure \ref{fig:disc_color}.
	\begin{figure}[t]
		\centering
		\includegraphics[scale=0.5]{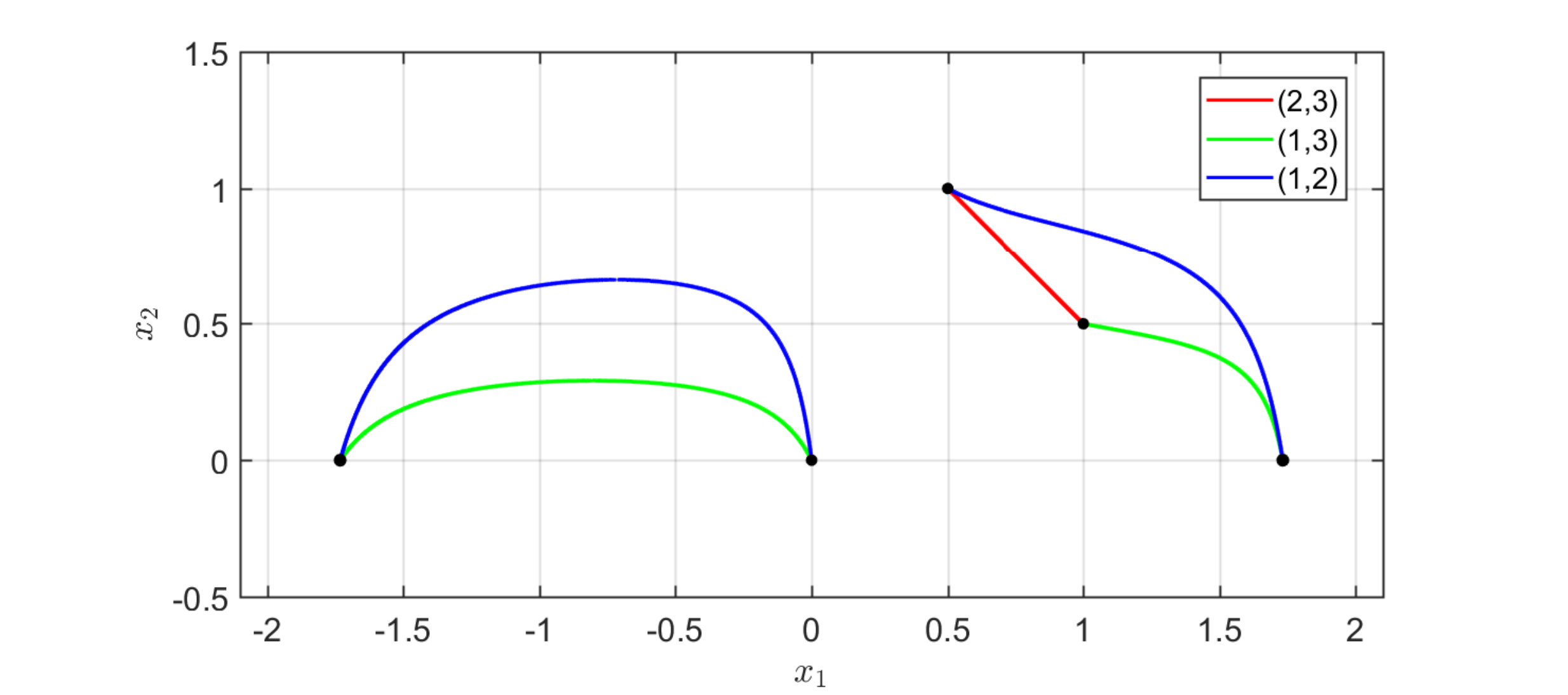}
		\caption{Pareto critical sets for the three 2-objective subproblems in Example \ref{ex:disc}.}
		\label{fig:disc_color}
	\end{figure}
\end{example}

\clearpage

\section{Conclusion and outlook} \label{sec:Conclusion}
\subsection{Conclusion}
We have presented results about the structure of the set of Pareto critical points with respect to the corresponding KKT multipliers. Our main result is that the boundary of the Pareto critical set can be covered by Pareto critical sets of subproblems in which we only consider a subset of the full objective function. The number of components required for the subproblems depends on the rank of the Jacobian of the objective function. To prove our main result we have investigated the relationship between tangent cones of the Pareto critical set and the tangent spaces of the manifold of Pareto critical points extended by their KKT multipliers. The boundary of the Pareto critical set can give useful insight into the global Pareto set or -- if it coincides with the topological boundary -- even describe it completely. 

\subsection{Outlook}
First of all there are some theoretical aspects that could be investigated further, for example the relationship between $\partial P$ and $P_E$ and what the requirements are such that $P_0 = P_E$. Additionally, since we only worked with a first order necessary optimality condition for Pareto optimality, it may be possible to extend our results by using nondominance tests or information about higher order derivatives. We also only considered the unconstrained case, so it will be interesting to see how our results convert to equality and inequality constrained MOPs. Finally, our results may be used to build new methods for solving MOPs via computation of the boundary of the Pareto critical set. A first approach in this direction has also been discussed in \cite{P2017}, where the well-known $\epsilon$-constraint method was generalized to considering subproblems with fewer objective functions instead of scalar problems.

\bibliography{literatur}
\bibliographystyle{abbrv}

\end{document}